\documentclass[a4paper,10pt]{scrartcl}

\usepackage{fullpage}

\usepackage{mathtools}
\usepackage{amsfonts, amsthm, amsmath, amssymb}
\usepackage{nicefrac}
\usepackage{xcolor}
\usepackage[shortlabels]{enumitem}
\usepackage{algorithm}
\usepackage{algpseudocode}
\usepackage{booktabs}
\definecolor{cb2blue}{RGB}{55,126,184}

\newcommand{\emailLink}[1]{\href{mailto:#1}{#1}}
\newcommand{\orcidLink}[1]{\href{https://orcid.org/#1}{#1}}
\newcommand{\myemail}[1]{\textsc{email} \emailLink{#1}}
\newcommand{\myorcid}[1]{\textsc{orcid} \orcidLink{#1}}
\newcommand{\amsmscLink}[1]{\href{http://www.ams.org/mathscinet/msc/msc2020.html?t=#1}{#1}}

\newcommand{\labelSPG}{\texttt{SPG}}
\newcommand{\labelRtwoN}{\texttt{R2N}}
\newcommand{\labelPANOC}{\texttt{PANOCplus}}
\newcommand{\labelRPQN}{\texttt{RPQN}}
\newcommand{\labelRPQNlkm}{\texttt{RPQN-lkm}}
\newcommand{\labelRPQNlbfgs}{\texttt{RPQN-lbfgs}}
\newcommand{\labelRPQNlsrone}{\texttt{RPQN-lsr1}}
\newcommand{\labelnm}{\texttt{-nm}}

\RequirePackage{doi}

\usepackage{hyperref}
\hypersetup{colorlinks,allcolors=cb2blue}
\usepackage[capitalize,noabbrev,nameinlink]{cleveref} 

\usepackage{multirow}

\algnewcommand{\IfThen}[2]{
	\State \algorithmicif\ #1\ \algorithmicthen\ #2}
\algnewcommand{\IfThenElse}[3]{
	\State \algorithmicif\ #1\ \algorithmicthen\ #2\ \algorithmicelse\ #3}
\newcommand{\lineref}[1]{\hyperref[#1]{Step~\ref{#1}}}

\theoremstyle{plain}
\newtheorem{theorem}{Theorem}[section]
\newtheorem{lemma}[theorem]{Lemma}
\newtheorem{proposition}[theorem]{Proposition}
\newtheorem{corollary}[theorem]{Corollary}
\theoremstyle{definition}
\newtheorem{assumption}{Assumption}
\newtheorem{definition}[theorem]{Definition}
\newtheorem{remark}[theorem]{Remark}

\crefname{assumption}{Assumption}{Assumptions}
\AddToHook{env/corollary/begin}{\crefalias{theorem}{corollary}}
\AddToHook{env/example/begin}{\crefalias{theorem}{example}}
\AddToHook{env/lemma/begin}{\crefalias{theorem}{lemma}}
\AddToHook{env/proposition/begin}{\crefalias{theorem}{proposition}}
\AddToHook{env/assumption/begin}{\crefalias{theorem}{assumption}}
\AddToHook{env/definition/begin}{\crefalias{theorem}{definition}}
\AddToHook{env/remark/begin}{\crefalias{theorem}{remark}}
\makeatletter
	\AddToHook{cmd/appendix/before}{\def\cref@section@alias{appendix}}
\makeatother

\AtBeginEnvironment{remark}{%
	\pushQED{\qed}%
}
\AtEndEnvironment{remark}{\popQED}

\newcommand{\R}{\mathbb{R}}
\newcommand{\Rinf}{\overline{\R}}
\newcommand{\N}{\mathbb{N}}
\newcommand{\Smat}{\mathbb{S}}
\newcommand{\solutionset}{\Omega}

\newcommand{\ared}{\mathrm{ared}}
\newcommand{\pred}{\mathrm{pred}}
\newcommand{\nmfactor}{\eta_{\textrm{nm}}}
\newcommand{\quadmodel}{\mathcal{Q}}

\DeclareMathOperator*{\minimize}{minimize}
\DeclareMathOperator{\dist}{dist}

\DeclareMathOperator{\prox}{prox}
\DeclareMathOperator{\localprox}{localprox}
\DeclareMathOperator*{\argmin}{argmin}
\DeclareMathOperator{\dom}{dom}

\DeclareMathOperator{\sign}{sign}
\DeclareMathOperator{\envelope}{env}
\newcommand{\ball}{\mathbb{B}}
\newcommand{\innerprod}[2]{\langle #1 , #2 \rangle}
\newcommand{\mumin}{\mu_{\min}}
\newcommand{\mumax}{\mu_{\max}}
\newcommand{\mubar}{\bar{\mu}}

\newcommand{\TheTitle}{Proximal Limited-Memory Quasi-Newton Methods for Nonsmooth Nonconvex Optimization}
\title{\bfseries\TheTitle}
\author{Simeon vom Dahl\thanks{%
    University of Würzburg, Institute of Mathematics, Emil-Fischer-Str.\ 30, 97074 Würzburg, Germany.}\and%
    Alberto De~Marchi\thanks{%
    University of the Bundeswehr Munich, Department of Aerospace Engineering, Institute of Applied Mathematics and Scientific Computing, Werner-Heisenberg-Weg 39, 85577 Neubiberg, Germany.\\
    \myemail{alberto.demarchi@unibw.de}, \myorcid{0000-0002-3545-6898}}\and%
    Christian Kanzow\thanks{%
    University of Würzburg, Institute of Mathematics, Emil-Fischer-Str.\ 30, 97074 Würzburg, Germany.\\
    \myemail{christian.kanzow@uni-wuerzburg.de}, \myorcid{0000-0003-2897-2509}}%
    }
\date{}

\begin{document}
\maketitle

\begin{abstract}
	We propose a regularized proximal quasi-Newton method (RPQN) for minimizing the sum of a continuously differentiable, possibly nonconvex function and a proper, lower semicontinuous, prox-bounded, possibly nonconvex regularizer.
	The method uses adaptive quadratic regularization and limited-memory quasi-Newton models.
	We establish global convergence under weak assumptions, without requiring global Lipschitz continuity of the gradient, and prove convergence of the entire sequence, with rates, under the Kurdyka--\L{}ojasiewicz property.
	To solve the proximal subproblems efficiently, we exploit compact limited-memory representations and derive a compact representation of the Kleinmichel rank-one update, which preserves positive definiteness under the same condition as BFGS. Numerical experiments on logistic and Student's regression demonstrate competitive performance and the effectiveness of the proposed approach.

    \medskip

    
    {\small
        \amsmscLink{90C06}, 
        \amsmscLink{90C26}, 
        \amsmscLink{90C53}. 
    }
\end{abstract}

\section{Introduction}

We consider the structured optimization problem
\begin{equation}\tag{P}\label{problem}
    \minimize_{x \in \R^n}~ F(x) \coloneqq f(x) + \varphi(x),
\end{equation}  
where $f \colon \R^n \to \R $ and $\varphi \colon \R^n \to \overline{\R}\coloneqq \R \cup \{\infty\}$ are given functions satisfying the following standing assumptions.
\begin{assumption}\label{assumption_general}
	The conditions below hold for \eqref{problem}:
    \begin{enumerate}[label=(\alph*)]
        \item The function $f$ is continuously differentiable on an open set containing $\dom \varphi$.  \label{ass_psi}
        \item The function $\varphi$ is proper, lower semicontinuous, and prox-bounded (\cref{def:prox_boundedness}). \label{ass:prox_bounded}
        \item The objective function $F\coloneqq f+\varphi$ is bounded from below, i.e., $\inf F > -\infty$. \label{ass_inf_F}
    \end{enumerate}
\end{assumption}
Under these conditions, $F \colon \R^n \to \overline{\R}$ is proper and lower semicontinuous, though it may be nonsmooth and nonconvex.

Structured optimization problems of the form \eqref{problem} arise across a wide range of applications in applied mathematics, particularly in high-dimensional settings where regularization is often indispensable.
In this setup, $f$ is called the \emph{loss function} and $\varphi$ the \emph{regularizer}.
Over the past few decades, researchers have proposed a wide range of regularizers and paired them with loss functions from many application areas.
The most classical regularizer is standard $\ell_1$-regularization.
It is often applied to linear regression---yielding the classical LASSO problem---or to logistic regression, both of which are ubiquitous in machine learning and statistics.
At the same time, interest in nonconvex regularizers has grown significantly in recent years; prominent examples include capped $\ell_1$-penalties, SCAD regularization, and MCP regularization.
Imaging applications such as image denoising \cite{rudin_1992} or deblurring form another major domain of interest \cite{bian_chen_2015}.
In addition to these practical applications, the framework \eqref{problem} includes the general problem of minimizing a continuously differentiable function over a closed, possibly nonconvex set.

\subsection{Related work}\label{subsec:relwork}

Classical proximal methods for \eqref{problem} are usually analyzed under convexity of the regularizer $\varphi$ and global Lipschitz continuity of $\nabla f$.
Many practically relevant models do not satisfy at least one of these assumptions, and a clear recent trend has therefore been to weaken them as much as possible.
At the first-order level, this relaxation happened in two stages.
First, convexity of the regularizer was abandoned:
after the proximal alternating linearized minimization (PALM) algorithm of Bolte, Sabach \& Teboulle~\cite{bolte2014proximal},
Li \& Lin~\cite{li_lin_2015} proved convergence of accelerated proximal-gradient schemes for nonconvex $\varphi$ under a globally Lipschitz gradient, and Themelis, Stella \& Patrinos~\cite{themelis_stella_patrinos_2018} derived the first forward-backward-type method with superlinear rates in the nonconvex setting.
Second, the global Lipschitz requirement on $\nabla f$ was removed in a sequence of proximal-gradient works~\cite{kanzow_mehlitz_2022,jia_kanzow_mehlitz_2023,demarchi2022proximal,demarchi_2023,kanzow_lehmann_2025}, in both monotone and nonmonotone variants.

These developments strongly motivate analogous weak assumptions for second-order methods.
Incorporating second‑order information into proximal algorithms leads to Newton‑type methods whose subproblems take the form
\begin{equation}\label{eq:PNk}
    \minimize_{x\in\R^n}~
    f(x^k) + \innerprod{\nabla f(x^k)}{x-x^k} + \frac{1}{2} \innerprod{x-x^k}{B_{k}(x-x^k)}+\varphi(x)
    ,
\end{equation}
where \(B_{k}\) is either the exact Hessian \(\nabla^{2}f(x^{k})\) or a (dense or limited‑memory) quasi-Newton approximation.

In the fully convex setting, Lee, Sun \& Saunders~\cite{lee_yuekai_saunders_2014} provided a thorough analysis of exact and inexact proximal Newton and proximal quasi-Newton algorithms.
Their theory delivers global convergence together with fast local rates, showing that these methods preserve the key results familiar from smooth optimization.
For large-scale applications, limited-memory quasi-Newton methods are usually the preferred choice, since storing full Hessians or dense quasi-Newton matrices is typically infeasible.
Becker \& Fadili \cite{becker_fadili_2012} proposed a limited-memory proximal quasi-Newton framework based on a zero-memory SR1 update, and Scheinberg \& Tang \cite{scheinberg_tang_2016} gave the first global-rate result for a practical L-BFGS-based scheme.

Growing attention has been recently devoted to substantially weaker global assumptions, similar to the developments for first-order methods.
Aravkin, Baraldi \& Orban \cite{aravkin_baraldi_orban_2022} were the first to allow nonconvex regularizers in a proximal Newton-type method, proving global convergence for a trust-region scheme under mere prox-boundedness of $\varphi$.
Most recently, Diouane, Habiboullah \& Orban \cite{diouane2026proximal} established global convergence results for a proximal modified quasi-Newton method without requiring convexity of $\varphi$, global Lipschitz continuity of $\nabla f$, or uniform boundedness of the model Hessians.
For smooth unconstrained optimization, Ueda \& Yamashita \cite{ueda_yamashita_2014} replaced classical line searches by a single regularization parameter updated through a success-ratio test.
Kanzow \& Lechner \cite{kanzow_lechner_2023} transferred this idea to nonsmooth problems and showed how the compact limited-memory representation of Byrd, Nocedal \& Schnabel \cite{byrd_nocedal_schnabel_1994} can be combined with the fast computation of quasi-Newton proximity operators from \cite{becker_fadili_ochs_2019}.

The symmetric rank-one (SR1) update offers inexpensive Sherman--Morrison formulas for both the matrix and its inverse, but it can lose positive definiteness.
Kleinmichel \cite{kleinmichel_1981_1,kleinmichel_1981_2} proposed a simple modification that preserves positive definiteness under the same condition as the BFGS formula while retaining rank-one updates and the cheap inverse formulas.
Despite these advantages, the update has seen little use.
Spellucci \cite{spellucci_2001} employed it in numerical tests against his own modified rank-one update.
However, to the best of our knowledge, no positive-definite rank-one update has been applied in proximal quasi-Newton methods before.

\subsection{Contributions}\label{subsec:contrib}

We propose a limited-memory regularized proximal quasi-Newton method (RPQN) for problem \eqref{problem} under \cref{assumption_general}. 
Our method builds on the algorithm introduced in \cite{becker_fadili_2012}, and, more specifically, on \cite{kanzow_lechner_2023, lechner_2022}.
However, their analysis assumes that $\varphi$ is convex and real-valued, and that $\nabla f$ is globally Lipschitz continuous.
Our work removes both limitations by allowing extended-valued nonconvex functions~$\varphi$ and requiring only local Lipschitz assumptions for convergence.
Thus, our assumptions are comparable to those in \cite{diouane2026proximal} and, under additional local Lipschitz gradient continuity, we prove stationarity of accumulation points.

Moreover, we consider limited-memory quasi-Newton approximations and show how to efficiently compute candidate updates for such models.
We present the first proximal-type method that systematically employs Kleinmichel's positive-definite rank-one update~\cite{kleinmichel_1981_1,kleinmichel_1981_2}.  
Although it retains the cheap Sherman--Morrison inverse of SR1 while guaranteeing positive definiteness, the formula has been largely overlooked. 
Extending the compact limited-memory framework of \cite{kanzow_lechner_2023, lechner_2022}, we formulate a limited-memory Kleinmichel update, prove its compact representation, and embed it in the semismooth Newton solver of \cite{kanzow_lechner_2023} for the proximal subproblem.
We show both theoretically and numerically that the same procedure can be applied for nonconvex regularizers too.
The resulting inner systems stay small, so the solver can be incorporated into RPQN with negligible overhead. 

\subsection{Organization}

\cref{Sec:Prelims} collects preliminaries.
\cref{Sec:Algorithm} presents the algorithm and its convergence theory: global convergence under minimal assumptions, convergence of the whole sequence under the KL property, and rates for the successful iterations.
\cref{sec_subproblems} shows how to solve the subproblems with compact limited-memory models, \cref{sec:kleinmichel} introduces the limited-memory Kleinmichel update, and \cref{sec:algo_refinements} describes implementation details.
Numerical results are reported in \cref{Sec:Numerics}, followed by concluding remarks.

\subsection{Notation}

Throughout, $\N = \{1,2, \dots \}$ denotes the set of positive integers and $\N_0 = \N\cup\{0\}$.
$\R$ is the real numbers and we write $\overline{\R}\coloneqq \R\cup\{\infty\}$ for the extended reals.
For $a,b\in\R$ we use $(a,b)$, $[a,b]$, $[a,b)$ and $(a,b]$ for the usual open, closed and semi-closed intervals. 

The sets
\[
\Smat^{n} \coloneqq 
\{A\in\R^{n\times n}\mid A^\top =A\}
\quad\text{and}\quad
\Smat_{++}^{n} \coloneqq
\{A\in \Smat^{n}\mid A\succ 0\}
\]
collect symmetric and symmetric positive definite matrices, respectively. The identity matrix (with dimensions clear from context) is $I$.

Vectors $x, y\in\R^n$ are equipped with the Euclidean norm $\|x\|$ and inner product $\innerprod{x}{y} \coloneqq x^\top y$.
For a matrix $H \in \Smat^n$, we let $\Vert x \Vert_H^2 \coloneqq \innerprod{x}{H x}$, which is the squared norm induced by $H$ whenever $H\in\Smat_{++}^{n}$.
For a closed set $C\subseteq\R^n$ the distance from $x$ to $C$ is $\dist(x,C)\coloneqq \inf_{y\in C}\|x-y\|$, and  
$\ball_\varepsilon(x)\coloneqq \{y\in\R^n \mid\|y-x\|\le \varepsilon\}$ denotes the closed ball of radius $\varepsilon$ around~$x$. 
Given an extended-valued function $f\colon\R^n \to\Rinf$, its domain is $\dom f\coloneqq \{x\in\R^{n}\mid f(x)<\infty\}$ and $f$ is called proper if $\dom f \neq \emptyset$.
Sequences are written $\{x^{k}\}_{k\in\mathcal I}$ for an index set $\mathcal I\subseteq\N_{0}$; we abbreviate this to $\{x^{k}\}_{\mathcal I}$, and simply $\{x^{k}\}$ when $\mathcal I=\N_{0}$.

\section{Preliminaries}\label{Sec:Prelims}
We collect here the background material needed for the sections that follow.

\subsection{Limiting subdifferential}

For an extended-valued function
\(F\colon\R^n \to \Rinf\) we denote by
\(\partial F(x)\) its \emph{limiting subdifferential} at~\(x\); full
definitions can be found in \cite{rock_wets_1998,Mor-18}.
If $F$ is continuously differentiable at $x$, then \(\partial F(x)=\{\nabla F(x)\}\).
If $F$ is convex, then \(\partial F(x)\) coincides with the classical convex subdifferential.
Otherwise, we only recall the facts that we explicitly use, all taken from \cite[Chapter 8]{rock_wets_1998}:
\begin{itemize}
    \item For the structured objective \(F=f+\varphi\) and every
    \(x\in\dom\varphi\),
    $\partial F(x)=\nabla f(x)+\partial\varphi(x)$.
    \item If $ \bar{x} $ is a local minimum of $ F $, then 
    $ 0 \in \partial F(\bar{x}) $.
    \item The mapping $\partial F$ is \emph{outer semicontinuous} with respect to $F$-attentive convergence, i.e., if $x^k\to x \in \dom F$, $F(x^k)\to F(x)$, and $v^k\in\partial F(x^k)$ with $v^k\to v$, then $v\in\partial F(x)$.
\end{itemize}
The third property is often also referred to as \emph{robustness} of the limiting subdifferential.
Motivated by the second property, a point \(x\) is called \emph{stationary} for \eqref{problem} if \(0\in\partial F(x)\); the set of all such points is denoted by $\solutionset$.
\cref{assumption_general}\ref{ass_psi}, together with the outer semicontinuity of \(\partial\varphi\), guarantees that $\solutionset$ is closed, although it may be empty in the present setting.

\subsection{Proximity operator}
Here we briefly recall the proximity operator and Moreau envelope, along with the notion of prox-boundedness, which is central to our analysis later. All results in this section are taken from \cite[Section 1.G]{rock_wets_1998}, where the reader is referred for a full treatment.

\begin{definition}\label{def:prox_operator}
	Let $\varphi \colon \R^n \to \Rinf$ be proper and lower semicontinuous, and let $\gamma > 0$.
	The \emph{Moreau envelope} and \emph{proximity operator} of $\varphi$ with parameter $\gamma$ are
	\begin{align*}
		\envelope_\gamma \varphi(x) \coloneqq{}& \inf_{y \in \R^n} \left\{ \varphi(y) + \frac{1}{2\gamma} \Vert y - x \Vert^2 \right\}, &
		\prox_{\gamma \varphi}(x) \coloneqq{}& \argmin_{y\in\R^n} \left\{ \varphi(y) + \frac{1}{2\gamma} \Vert y-x \Vert^2 \right\}.
	\end{align*}
\end{definition}
In general, $\prox_{\gamma\varphi}$ is set-valued, whereas for convex $\varphi$ it is single-valued.
If $\varphi$ is an indicator function $\delta_C$, then $\prox_{\gamma\varphi}$ coincides with the projection onto $C$.
To formulate existence properties of proximal points in this general setting, we next recall prox-boundedness.

\begin{definition}\label{def:prox_boundedness}
A function $\varphi \colon \R^n \to \Rinf$ is called \emph{prox-bounded} if there exists $\gamma > 0$
such that $\envelope_\gamma \varphi(x) > -\infty$ for some $x \in \R^n$.
The supremum of all such $\gamma$ is called the \emph{threshold of prox-boundedness} of $\varphi$ and is denoted by $\gamma_\varphi$.
\end{definition}
The characterizations given in \cite[Exercise 1.24]{rock_wets_1998} show that prox-boundedness is clearly a mild condition satisfied by a broad class of functions. In particular, every proper, lower semicontinuous function that is bounded from below is also prox-bounded (with $\gamma_\varphi = +\infty$).

The following proposition summarizes key regularity properties of the Moreau envelope and proximal mapping under prox-boundedness.

\begin{proposition}[{\protect\cite[Theorem 1.25]{rock_wets_1998}}]\label{prop:prox_boundedness_properties}
Let $\varphi \colon \R^n \to \Rinf$ be proper, lower semicontinuous, and prox-bounded. Then for every $\gamma \in (0,\gamma_\varphi)$, the set $\prox_{\gamma\varphi}(x)$ is nonempty and compact for all $x\in\R^n$, whereas the value $\envelope_\gamma \varphi(x)$ is finite and depends continuously on $(\gamma,x)$.
\end{proposition}

Note that \cref{prop:prox_boundedness_properties} sharpens \cref{def:prox_boundedness} in an important way: for every $\gamma\in(0,\gamma_\varphi)$, the envelope is finite \emph{for all} $x\in\R^n$.
Hence, the qualifier \emph{some} in \cref{def:prox_boundedness} can effectively be read as \emph{all}.

\subsection{Kurdyka--\L{}ojasiewicz property}
We now state the definition of the celebrated Kurdyka--\L{}ojasiewicz (KL) property, which will play a central role in our subsequent convergence analysis.
Since the nonsmooth version was introduced in \cite{bolte_2007}, it has been invoked countless times as a key assumption in the local convergence theory of algorithms for problem \eqref{problem}.

\begin{definition}\label[definition]{def_KL}
    Let $g \colon \R^n \to \overline{\R}$ be proper and lower semicontinuous.
    We say that $g$ has the \emph{KL property} at $x^* \in \{ x \in \R^n \,\vert\, \partial g (x) \neq \emptyset \}$ if there exist a constant $\eta > 0$, a neighborhood $U \subset \R^n$ of $x^*$, and a continuous concave function $\chi \colon [0,\eta] \to [0,\infty)$ which is continuously differentiable on $(0,\eta)$ and satisfies $\chi (0) = 0$ as well as $\chi'(t) > 0$ for all $t \in (0,\eta)$ such that the so-called \emph{KL inequality}
    \[
    \chi'\left(g(x)-g(x^*)\right) \dist\left( 0, \partial g(x) \right) 
    \geq 1
    \]
    holds for all $x \in U \cap \left\{x \in \R^n \, \middle\vert \, g(x^*) < g(x) < g(x^*) + \eta \right\}$.
    The function $\chi$ from above is referred to as the \emph{desingularization function}.
    In case it can be chosen to be the function $t \mapsto ct^{1-\theta}$ with $\theta \in [0,1)$ for some $c>0$, then $g$ is said to have the KL property of exponent $\theta$ at $x^*$.
    If $g$ has the KL property (of exponent $\theta$) everywhere, it is called a KL function (of exponent $\theta$).
\end{definition}

By \cite[Lemma 2.1]{attouch_2010}, every proper, lower semicontinuous function $g \colon \R^n \to \overline \R$ already enjoys the KL property at all non-critical points $x$, i.e., those with $0 \notin \partial g(x)$.
Subsequent work has identified functions that satisfy the property everywhere.
In particular, Bolte et al. \cite{bolte_2007} showed that all \emph{tame} functions---those definable on o-minimal structures---are KL.
This covers real polynomials, $p$-norms, exponentials, logarithms, and many others.
Because tame functions are closed under finite sums and compositions, this yields a large family of KL functions \cite{attouch_2010}.

To simplify the KL-based convergence-rate analysis later on, we will make use of the following technical lemma from \cite[Lemma 1]{artacho_fleming_vuong_2018}.

\begin{lemma}\label{lemma_technical_rates}
    Let $\{a_j\}_{j\in\N}$ be a non-negative, monotonically decreasing
    sequence converging to $0$, and suppose that for all
    sufficiently large $j$,
    \begin{equation}\label{eq:lemma-recursion}
        a_j^\alpha \leq \beta (a_j - a_{j+1}),
    \end{equation}
    for some constants $\alpha,\beta>0$. Then the following statements hold:
    \begin{enumerate}[label=\textup{(\roman*)}]
        \item If $\alpha\in(0,1]$, then $\{a_j\}$ converges linearly to $0$ with rate $1 - \tfrac{1}{\beta}$.
        \item If $\alpha>1$, then there exists a constant $C>0$ such that
        \[
        a_j \leq C j^{-\frac{1}{\alpha-1}}
        \qquad\text{for all $j$ sufficiently large.}
        \]
    \end{enumerate}
\end{lemma}

\section{Methodology and convergence analysis}\label{Sec:Algorithm}

This section introduces our regularized proximal quasi-Newton method and presents convergence results under a hierarchy of assumptions.

\subsection{Scaled and local proximity operator}
The proximity operator in \cref{def:prox_operator} is naturally extended by replacing $\frac{1}{2\gamma}\|\cdot\|^2$ with $\frac12\|\cdot\|_H^2$ for a symmetric matrix $H \in \Smat^n$.
Such \emph{scaled proximity operator} $\prox_\varphi^H\colon \R^n\rightrightarrows\R^n$ is defined by
\begin{equation}\label{eq:scaled_prox_operator}
	\prox_\varphi^H(x)
	\coloneqq
	\argmin_{y\in\R^n}\left\{
	\varphi(y) + \frac{1}{2} \|y-x\|_H^2
	\right\}.
\end{equation}
The original scalar definition is recovered by $H = \gamma^{-1}I$.
Denoting by
\begin{equation}\label{eq:weighted_prox_model}
	\quadmodel(y,x,H) \coloneqq f(x)+\varphi(y)+\innerprod{y-x}{\nabla f(x)}+\frac{1}{2}\|y-x\|_H^2
\end{equation}
a local quadratic, possibly nonconvex model of $F$ around $x$,
the proximal mapping given by \cref{def:prox_operator} leads to
\[
\prox_{\gamma \varphi}(x-\gamma\nabla f(x)) 
=
\argmin_{y\in\R^n} \left\{ \varphi(y)+\frac{1}{2\gamma}\|y-x+\gamma\nabla f(x)\|^2 \right\}
=
\argmin_{y\in\R^n} \quadmodel(y,x,\gamma^{-1} I) ,
\]
which is the classical proximal-gradient update.
This problem can be efficiently solved, even in closed-form, for many regularizers $\varphi$, thanks to the simple Hessian approximation $H=\gamma^{-1} I$ with $\gamma>0$.
In general, this is not the case for non-diagonal matrices $H$, and it becomes unrealistic to assume that a global minimizer of $\quadmodel(\cdot,x,H)$ can be approximately computed.
Therefore, following \cite[\S 4]{demarchi2022proximal}, we consider a weaker proximal oracle that requires simple decrease and stationarity (although the latter could be relaxed as well, we focus on exact stationarity for the sake of simplicity).

\begin{definition}\label{def:local_prox_operator}
	Let $f\colon \R^n\to\R$ be differentiable,
	$\varphi \colon \R^n \to \Rinf$ be proper and lower semicontinuous,
	$H \in \Smat^n$,
	and let $\quadmodel(\cdot,\cdot,H)$ be defined in \eqref{eq:weighted_prox_model}.
	The \emph{local proximity operator} of $\varphi$ (relative to $f$ and scaled with $H$) is the mapping $\localprox_{f,\varphi}^H\colon \R^n \rightrightarrows \dom\varphi$ defined by
	\begin{equation*}
		\localprox_{f,\varphi}^H(x)
		\coloneqq
		\left\{
		z \in\R^n \middle\vert
		\begin{array}{l}
			\quadmodel(z,x,H) \leq \quadmodel(x,x,H) \\
			0 \in \partial [\quadmodel(\cdot,x,H)](z)
		\end{array}
		\right\}
		.
	\end{equation*}
\end{definition}
By construction, these concepts are related according to the following result.
\begin{lemma}\label{lem:scaled_prox_relation}
	Let $x\in\R^n$ and $H\in\Smat_{++}^n$.
	Then
	\[
	\prox_\varphi^H \left( x - H^{-1}\nabla f(x) \right)
	=
	\argmin \quadmodel(\cdot,x,H)
	\subseteq
	\localprox_{f,\varphi}^H\left(x\right)
	\subseteq
	\dom \varphi.
	\]
\end{lemma}
\begin{proof}
	By $H\in\Smat_{++}^n$ and the definition of scaled proximal operator \eqref{eq:scaled_prox_operator}, we have
	\begin{align*}
		\prox_\varphi^H{}&{}\left( x - H^{-1}\nabla f(x) \right) = \argmin_{y \in \R^n} \left\{ \varphi(y) + \frac12 \big\Vert y - (x - H^{-1} \nabla f(x)) \big\Vert_H^2 \right\} \\
		={}&
		\argmin_{y \in \R^n} \left\{ \varphi(y) + \frac12 \innerprod{y - x + H^{-1} \nabla f(x) }{H \left( y - x + H^{-1} \nabla f(x) \right)} \right\} \\
		={}&
		\argmin_{y \in \R^n} \left\{ \varphi(y) + \innerprod{\nabla f(x)}{y-x} + \frac{1}{2} \|y-x\|_H^2 \right\}
		={}
		\argmin_{y \in \R^n}~ \quadmodel(y,x,H)
		\subseteq
		\localprox_{f,\varphi}^H(x),
	\end{align*}	
	where we repeatedly used the fact that adding a constant to the objective function does not change its minimizer.
	The last inclusion is due to \cref{def:local_prox_operator} and Fermat's rule.
\end{proof}

\subsection{Algorithm}\label{sec:algorithm}

Consider a fixed iteration $k\in\N_0$, and let $x^k\in\R^n$ denote the current iterate.
Proximal Newton--type methods compute the next iterate by solving the subproblem
\[
    \minimize_{x\in\R^n}\quad
    q_k(x) \coloneqq f(x^{k}) + \innerprod{\nabla f(x^k)}{x-x^k} + \frac{1}{2} \|x-x^k\|_{B_k}^2 + \varphi(x),
\]
where the symmetric matrix \(B_{k}\) captures second-order information of \(f\).
A \emph{proximal Newton} method takes \(B_{k}=\nabla^{2}f(x^{k})\), whereas a \emph{proximal quasi-Newton} method uses an approximation \(B_{k}\approx\nabla^{2}f(x^{k})\).
Because, in general, $\varphi$ is nonconvex and the Hessian surrogate $B_k$ may be indefinite, $q_k$ is not necessarily convex.  
We apply a positive spectral shift
\begin{equation}\label{eq:Gk}
    G_k \coloneqq B_k+\mu_k I,\qquad \mu_k>0,
\end{equation}
and instead solve the regularized subproblem
\begin{align}\label{eq_subproblem}
    \minimize_{x\in\R^n}\quad
    \hat{q}_k(x)
    \coloneqq
    f(x^k)+\innerprod{\nabla f(x^k)}{x-x^k} + \frac{1}{2} \|x-x^k\|_{G_k}^2 + \varphi(x).
\end{align}
The basic idea of our method is to find, in each iteration, a minimizer of \eqref{eq_subproblem}.
Since this is rather demanding in general, due to the possible nonconvexity of the underlying problem, we search only for a stationary point whose function value is not worse than the current one, i.e., we are looking for a point $\hat{x}^k$ satisfying
\begin{align}
	\hat{q}_k(\hat{x}^k) \leq \hat{q}_k(x^k),
	\qquad	
	0 \in \partial \hat{q}_k(\hat{x}^k).
	\label{eq:localprox}
\end{align}
By comparing \eqref{eq_subproblem} with \eqref{eq:weighted_prox_model}, it follows from \cref{def:local_prox_operator} that $\hat{x}^k \in \localprox_{f,\varphi}^{G_k}(x^k)$, namely it is a \emph{local} proximal quasi-Newton update.
Although this is a much weaker requirement than the computation of a global minimum of \eqref{eq_subproblem}, it may still happen that such a point does not exist, especially in the case where $ G_k $ is not positive definite.
In this (unsuccessful) situation, we simply enlarge the regularization parameter $\mu_k$.
Otherwise, we write
\begin{equation}\label{eq:d_definition}
	d^k \coloneqq \hat x^k - x^k.
\end{equation}

Our globalization strategy adaptively updates the regularization parameter $\mu_k$.
This approach was introduced in \cite{ueda_yamashita_2014} and extended to nonsmooth problems in \cite{lechner_2022,aravkin_baraldi_orban_2022,diouane2026proximal}.
To evaluate the candidate iterate $\hat x^{k}$, we compare the actual decrease in the objective with the decrease predicted by the quadratic model, namely
\begin{equation}\label{eq:ared-pred-k}
    \ared_k \coloneqq F(x^{k}) - F(\hat x^{k}), 
    \qquad
    \pred_k \coloneqq F(x^{k}) - q_k(\hat x^{k}).
\end{equation}

\begin{algorithm}[tbh]
    \caption{Regularized Proximal Quasi-Newton Method} 
    \label{alg:RPQN}
    \begin{algorithmic}[1]
    	\State Choose parameters $c_1 \in (0,1)$; $\sigma_2>1$; $0 < \mumin \leq \mumax < \infty$.
        \State Choose $x^0 \in \dom\varphi$, $B_0 \in \Smat^n$, $\mu_0 \in [\mumin, \mumax]$.
        \For{$k=0,1,2,\ldots$}
        \State Set $G_k = B_k + \mu_k I$.
        \State Search for a solution $\hat x^k$ of the regularized subproblem \eqref{eq_subproblem}, in the sense of \eqref{eq:localprox}.\label{algo_line_subproblem}
        \IfThen{$\hat x^k$ is found}{compute $\ared_k$ and $\pred_k$ as in \eqref{eq:ared-pred-k}.\label{algo_line_ared_pred}}
        \If{$\hat x^k$ is found \textbf{and} $\ared_k \geq c_1 \pred_k$\label{algo_line_acceptance_test}}
            \State Set $x^{k+1} = \hat x^k$ and choose $B_{k+1} \in \Smat^n$, $\mu_{k+1} \in [ \mumin, \mumax ]$. \label{algo_line_K_commands_2}\Comment{successful}
        \Else
            \State Set $x^{k+1}=x^k$, $B_{k+1} = B_k$ and $\mu_{k+1}=\sigma_2\mu_k$.\Comment{unsuccessful}
        \EndIf
        \EndFor
    \end{algorithmic}
\end{algorithm}

At every iteration \cref{alg:RPQN} tries to obtain a solution of the regularized proximal quasi-Newton subproblem~\eqref{eq_subproblem}.
If a candidate is found and passes the acceptance tests in \lineref{algo_line_acceptance_test}, then it is accepted as the next iterate; otherwise the current point is retained and the regularization parameter is increased.

The update of $\mu_k$ is not prescribed uniquely: \lineref{algo_line_K_commands_2} corresponds to picking any element from a user-defined compact (yet arbitrarily large) interval $[\mumin,\mumax]$.
The upper bound $\mumax$ guarantees boundedness of all $\mu_k$ following a successful iteration.
Later, in \cref{lemma:mu_bounded}, we will see that this transfers to a crucial local boundedness property for $\mu_k$.

Note that $B_k$ is defined in a way that it remains constant in unsuccessful iterations.
This is natural since $ B_k $ gives an approximation of the (not necessarily existing) Hessian of $ f $ at $ x^k $, so that no change in $ x^k $ implies no change in the approximation $ B_k $.

In what follows, we denote by
\begin{align*}
	\mathcal{K} \coloneqq{}& \left\{
		k \in \N_0 \,\middle\vert\, \hat{x}^k \text{ is found at \lineref{algo_line_subproblem}}
	\right\},\\
	\mathcal{S} \coloneqq{}& \left\{
		k \in \mathcal{K} \,\middle\vert\, \text{iteration } k \text{ is successful}
	\right\}, \\
	\mathcal{U} \coloneqq{}& \left\{
		k \in \mathcal{K} \,\middle\vert\, \text{iteration } k \text{ is unsuccessful}
		\right\},
\end{align*}
the sets of \emph{computed}, \emph{successful}, and \emph{unsuccessful} iterations of \cref{alg:RPQN}.

\subsection{Global convergence}

The remainder of this section is devoted to the global convergence analysis of \cref{alg:RPQN}, whose proof structure follows \cite{diouane2026proximal}.
Throughout our convergence analysis, we make the implicit assumption that,
for all iterations $k \in \N_0$, a vector $\hat x^k$ satisfying \eqref{eq:localprox} is found at \lineref{algo_line_subproblem} of \cref{alg:RPQN} whenever subproblem \eqref{eq_subproblem} admits one.

Using the definition of $q_k$, $\hat q_k$, and $d^k$, we now establish an explicit lower estimate for $\pred_k$ together with two criteria that guarantee stationarity of $x^k$.

\begin{lemma}\label{lemma:pred}
    Suppose that \cref{assumption_general} holds.
    For all $k \in \mathcal K$ it holds that $\pred_k \geq \frac{\mu_k}{2} \Vert d^k \Vert^2$.
    Furthermore, we have
    \[
    \pred_k = 0 \quad \implies \quad  \Vert d^k \Vert = 0  \quad \implies  \quad x^k \text{ is a stationary point of \eqref{problem}.}
    \]
\end{lemma}
\begin{proof}
    By definition, it holds that $\hat q_k(x^k)=F(x^k)$.
    Hence, because $\hat x^k$ satisfies $\hat q_k(\hat{x}^k) \leq \hat q_k(x^k)$ by \eqref{eq:localprox}, we have $F(x^k)-\hat q_k(\hat x^k) \geq F(x^k)-\hat q_k(x^k) = 0$.
    Using the identity $q_k(x)=\hat q_k(x)-\frac{\mu_k}{2}\|x-x^k\|^2$, we therefore get
	\begin{equation}\label{eq_pred}
		\pred_k
		= F(x^k)-q_k(\hat x^k)
		= F(x^k)-\hat q_k(\hat x^k)+\frac{\mu_k}{2}\Vert d^k \Vert^2
		\geq \frac{\mu_k}{2}\Vert d^k \Vert^2.
	\end{equation}
    Since $\mu_k > 0$, $\pred_k = 0$ implies $\Vert d^k \Vert = 0$.
    Then using \eqref{eq:d_definition} and again \eqref{eq:localprox}, it follows that
    \[
    0 \in \partial \hat q_k(\hat x^k) = \partial \hat q_k(x^k) = \nabla f(x^k) + \partial \varphi(x^k) = \partial F(x^k).
    \]
    Hence, $x^k$ is a stationary point of \eqref{problem}.
\end{proof}

Together with the update mechanism in \cref{alg:RPQN}, the previous lemmas imply three basic properties that will be used repeatedly in the convergence analysis.

\begin{lemma}\label{lemma:basic_properties}
    Suppose that \cref{assumption_general} holds.
    Then \cref{alg:RPQN} satisfies the following properties:
    \begin{enumerate}[label=(\alph*)]
      \item\label{lemma:basic_mu} For all $k \in \N_0$ it holds that $\mu_k \geq \mumin$.
      \item\label{lemma:basic_pred} For all $k \in \mathcal K$ it holds that $\pred_k \geq \frac{\mumin}{2} \Vert d^k \Vert^2$.
      \item\label{lemma:basic_F} The sequence $\{ F(x^k) \}$ is monotonically decreasing.
    \end{enumerate}
\end{lemma}
\begin{proof}
    Statement \ref{lemma:basic_mu} follows inductively from the possible updates for $\mu_k$ in \cref{alg:RPQN}. Statement \ref{lemma:basic_pred} then simply follows from \cref{lemma:pred} and \ref{lemma:basic_mu}.
    Let $k \in \N_0$.
    If $k$ is an unsuccessful iteration, then it simply holds that $F(x^{k+1}) = F(x^k)$.
    If $k$ is successful, then it follows from \ref{lemma:basic_pred} that
    \[
    F(x^k)-F(x^{k+1})
    =
    \ared_k
    \geq
    c_1 \pred_k
    \geq
    0.
    \]
    This shows that the sequence $\{F(x^k)\}$ is monotonically decreasing.
\end{proof}

Our global convergence analysis is organized with a hierarchy of regularity assumptions.
Starting from \cref{assumption_general}, we first establish elementary properties of the computed steps $d^k$ and arrive at \cref{cor:liminf_dk}.
We then incorporate boundedness of the matrices $\{B_k\}$ to obtain a true stationarity conclusion; see \cref{thm:liminf_stationarity}.
In the final step, a local Lipschitz assumption on $\nabla f$ allows us to promote these results to stationarity of all accumulation points; see \cref{theorem:global_convergence_1}.

\begin{lemma}\label{theorem:global_convergence_cvx}
    Suppose that \cref{assumption_general} holds and \cref{alg:RPQN} performs infinitely many successful iterations.
    Then it holds that $\lim\limits_{k \in \mathcal S} \Vert d^k \Vert = 0$.
\end{lemma}
\begin{proof}
    By definition, for every $k\in\mathcal S$,
    \[
    F(x^k)-F(\hat x^k)
    =
    \ared_k
    \geq
    c_1 \pred_k
    \geq
    \frac{c_1 \mumin}{2} \Vert d^k \Vert^2,
    \]
    where we used \cref{lemma:basic_properties}\ref{lemma:basic_pred} in the last inequality.
    Since $F$ is bounded from below by \cref{assumption_general}, summation over $k$ yields
    \[
    \infty
    >
    \sum_{k=0}^{\infty}\bigl[F(x^k)-F(x^{k+1})\bigr]
    =
    \sum_{k\in\mathcal S}\bigl[F(x^k)-F(\hat x^k)\bigr]
    \geq
    \frac{c_1 \mumin}{2}\sum_{k\in\mathcal S}\|d^k\|^2.
    \]
    Therefore $\lim_{k \in \mathcal S}\|d^k\|=0$. 
\end{proof}

The following result concerns the case with finitely many successful iterations.
\begin{lemma}\label{lemma:q_hat}
    Suppose that \cref{assumption_general} holds.
    The set $\mathcal K$ is infinite and for every index set $\mathcal L \subset \N_0$ with $\{x^k\}_{\mathcal L}$ and $\{B_k\}_{\mathcal L}$ bounded and $\{\mu_k\}_{\mathcal L} \to \infty$, the set $\mathcal L' \coloneqq \mathcal L \cap \mathcal K$ is also infinite and it holds that $\{ d^k \}_{\mathcal L'} \to 0$.
\end{lemma}
\begin{proof}
    By \cref{def:prox_boundedness,prop:prox_boundedness_properties}, the Moreau envelope $\envelope_\gamma \varphi$ is real-valued and continuous and satisfies, for all sufficiently small $\gamma>0$, $\envelope_\gamma \varphi(z) > -\infty$ for all $z \in \R^n$.
    Hence, $\varphi(x) \geq \envelope_\gamma \varphi(x^k) - \frac{1}{2\gamma}\Vert x - x^k \Vert^2$ for all $x \in \R^n$, from which we obtain
    \begin{align}\label{eq:q_hat}
	    \hat q_k(x)
	    ={}&
	    q_k(x) + \frac{\mu_k}{2} \Vert x-x^k \Vert^2
	    \nonumber\\
	    ={}&
	    f(x^k) + \nabla f(x^k)^\top (x-x^k) + \frac{1}{2} \|x-x^k\|_{B_k}^2 + \varphi(x) + \frac{\mu_k}{2} \Vert x-x^k \Vert^2
	    \nonumber\\
	    \geq{}&
	    f(x^k) + \nabla f(x^k)^\top (x-x^k) + \frac{1}{2} \|x-x^k\|_{B_k}^2
        + \frac{1}{2} \left( \mu_k - \frac{1}{\gamma} \right) \Vert x-x^k \Vert^2 + \envelope_\gamma \varphi(x^k)
	    \\
	    \geq{}&
	    f(x^k) - \Vert \nabla f(x^k) \Vert \Vert x-x^k \Vert - \frac{1}{2} \Vert B_k \Vert \Vert x-x^k \Vert^2 + \frac{1}{2} \left( \mu_k - \frac{1}{\gamma} \right) \Vert x-x^k \Vert^2 + \envelope_\gamma \varphi(x^k)
	    \nonumber\\
	    ={}&
	    f(x^k) - \Vert \nabla f(x^k) \Vert \Vert x-x^k \Vert + \frac{1}{2} \left( \mu_k - \frac{1}{\gamma} - \Vert B_k \Vert \right) \Vert x-x^k \Vert^2 + \envelope_\gamma \varphi(x^k), 
    \end{align}
    where we used prox-boundedness of $\varphi$ in the first inequality, and the Cauchy-Schwarz inequality in the second.
    This shows that for $\mu_k > \frac{1}{\gamma}+\Vert B_k \Vert$, $\hat q_k$ is coercive and together with lower semi-continuity of $\hat q_k$ it follows by Weierstrass' theorem that $\hat q_k$ has at least one minimizer.
    In particular, \cref{alg:RPQN} therefore finds a point $\hat{x}^k$ in the sense that \eqref{eq:localprox} holds.
    Assume, by contradiction, that $\mathcal K$ is only a finite set. Then all iterates are eventually unsuccessful.
    However, then it must hold that $\mu_k > \frac{1}{\gamma}+\Vert B_k \Vert$ for $k$ sufficiently large, a contradiction.
    Hence, $\mathcal K$ is an infinite set.
    
    For an index set $\mathcal L \subset \N_0$ with $\{x^k\}_{\mathcal L}$ and $\{B_k\}_{\mathcal L}$ bounded and $\{\mu_k\}_{\mathcal L} \to \infty$, it follows immediately from \eqref{eq:q_hat} that $k \in \mathcal L'$ for $k \in \mathcal L$ sufficiently large.
    For all $k \in \mathcal L'$ it holds that $F(x^k) = \hat q_k(x^k) \geq \hat q_k(\hat x^k)$ by \eqref{eq:localprox}, which together with \eqref{eq:q_hat} implies $\{d^k\}_{\mathcal L'} \to 0$, as otherwise the aforementioned properties of $\mathcal L$ together with continuity of $\nabla f$ and  $\envelope_\gamma \varphi$ would yield a contradiction to $F(x^k) \leq F(x^0)$, which follows from \cref{lemma:basic_properties}\ref{lemma:basic_F}.
\end{proof}

\begin{corollary}\label{cor:liminf_dk}
    Suppose that \cref{assumption_general} holds. Then it holds that
    $\liminf_{k \in \mathcal K} \Vert d^k \Vert = 0$.
\end{corollary}

\begin{proof}
    This follows immediately from \cref{theorem:global_convergence_cvx} for the case with infinitely many successful iterations,
    and from \cref{lemma:q_hat} otherwise.
\end{proof}

A small step norm $\Vert d^k \Vert$ is often used in practice as a stopping criterion, and \cref{lemma:pred} explains why this is reasonable in the present setting.
Accordingly, the previous corollary can already be viewed as a (weak) global convergence statement under the basic assumptions alone, in particular, without assuming boundedness of the sequence $\{B_k\}$.
That said, vanishing step lengths are not the same as a quantitative stationarity estimate: even though $\Vert d^k \Vert = 0$ forces $x^k$ to be stationary, one cannot control $\dist(0,\partial F(x^k))$ by $\Vert d^k \Vert$ without extra structure.
For this reason, we now strengthen the assumptions.

\begin{assumption}\label{ass:Bk_bounded}
	The sequence $\{B_k\}$ is bounded, i.e., there exists $M_B > 0$ with $\Vert B_k \Vert \leq M_B$ for all $k \in \N_0$.    
\end{assumption}

We now present two technical results ahead of \cref{theorem:mu_pred_3}.

\begin{lemma}\label{lemma:mu_pred_2}
    Suppose that \cref{assumption_general,ass:Bk_bounded} hold.
    For an infinite set $\mathcal L \subset \mathcal U$, if $\{x^k\}_{\mathcal L}$ is bounded and $\{\mu_k\}_{\mathcal L} \to \infty$, then it holds that $\lim\limits_{k \in \mathcal L} \mu_k \Vert d^k \Vert = 0$.
\end{lemma}
\begin{proof}
	Let us assume, by contradiction, that there exists some $\omega > 0$ such that $\mu_k \Vert d^k \Vert \geq \omega$ on a subset $\mathcal L' \subset \mathcal L$. 
	From \cref{lemma:q_hat} we know that $\{d^k\}_{\mathcal L'} \to 0$, whereas \eqref{eq_pred} gives
    \begin{equation}\label{eq_muk_dk_1}
    \pred_k \geq \frac{\mu_k}{2} \Vert d^k \Vert^2 \geq \frac{\omega}{2} \Vert d^k \Vert
    \end{equation}
    for all $k \in \mathcal L'$.
    By Taylor's theorem, there exists $\xi^k$ on the line segment between $x^k$ and $\hat x^k$ with $f(\hat x^k) = f(x^k) + \nabla f(\xi^k)^\top d^k$.
    From the boundedness of $\{x^k\}_{\mathcal L'}$ and $\{d^k\}_{\mathcal L'} \to 0$ it follows that both $x^k$ and $\xi^k$ belong to a sufficiently large compact set for all $k \in \mathcal L'$.
    The continuity of $\nabla f$ implies uniform continuity on this compact set and, hence, from $\{\Vert \xi^k - x^k \Vert\}_{\mathcal L'} \to 0$ it follows that
    \begin{equation}\label{eq:uniform_conv}
    \Vert \nabla f(\xi^k) - \nabla f(x^k) \Vert \to_{\mathcal L'} 0.
    \end{equation}
    Therefore, denoting $\rho_k \coloneqq \ared_k / \pred_k$ we get
    \begin{align*}
    	\vert \rho_k - 1 \vert
    	={}&
    	\left\vert \frac{\ared_k - \pred_k}{\pred_k}\right\vert = \frac{\left\vert f(\hat x^k) - f(x^k) - \nabla f(x^k)^\top d^k - \frac{1}{2} \|d^k\|_{B_k}^2 \right\vert}{\pred_k}
    	\\
    	\leq{}&
    	\frac{\left\vert f(\hat x^k) - f(x^k) - \nabla f(x^k)^\top d^k \right\vert}{\pred_k} + \frac{\left\vert \|d^k\|_{B_k}^2 \right\vert}{2 \pred_k}
    	={}
    	\frac{\left\vert (\nabla f(\xi^k) - \nabla f(x^k))^\top d^k \right\vert}{\pred_k} + \frac{\left\vert \|d^k\|_{B_k}^2 \right\vert}{2 \pred_k}
    	\\
    	\leq{}&
    	\frac{\Vert \nabla f(\xi^k) - \nabla f(x^k) \Vert \Vert d^k \Vert}{\pred_k} + \frac{\Vert B_k \Vert \Vert d^k \Vert^2}{2 \pred_k}
    	\leq
    	\frac{2\Vert \nabla f(\xi^k) - \nabla f(x^k) \Vert}{\omega} + \frac{\Vert B_k \Vert \Vert d^k \Vert}{\omega}
    \end{align*}
    for all $k \in \mathcal L'$ sufficiently large, where we used the Cauchy-Schwarz inequality in the second, and \eqref{eq_muk_dk_1} in the third inequality.
    By $\{d^k\}_{\mathcal L'} \to 0$, boundedness of $\{B_k\}$ and \eqref{eq:uniform_conv}, it follows that the right-hand side converges to zero, which implies that $\{\rho_k\}_{\mathcal L'} \to 1$.
    Hence, it holds that $\ared_k \geq c_1 \pred_k$ for $k \in \mathcal L'$ sufficiently large. This yields a contradiction as such a $k$ would be a successful iteration.
\end{proof}

An immediate consequence is the finite-success case: if only finitely many iterations are successful, then both $\{x^k\}$ and $\{B_k\}$ eventually stop changing, while $\mu_k \to \infty$, and therefore $\{d^k\}_{\mathcal K} \to 0$.

\begin{lemma}\label{lemma:xk_unbounded}
    Suppose \cref{assumption_general} holds.
    If $\{x^k\}$ is unbounded, then $\liminf\limits_{k \in \mathcal K} \mu_k \Vert d^k \Vert = 0$.
\end{lemma}
\begin{proof}
	We prove the statement by contraposition.
    Suppose that there exists $\omega > 0$ with $\mu_k \Vert d^k \Vert \geq \omega$ for all sufficiently large $k \in \mathcal K$.
    Then it follows that
    \begin{multline*}
    \infty
    >{}
    F(x^0) - \inf F
    \geq
    \sum_{k=0}^\infty F(x^k) - F(x^{k+1})
    =
    \sum_{k \in \mathcal S} F(x^k) - F(x^{k+1})
    \\
    \geq{}
    c_1 \sum_{k \in \mathcal S} \pred_k
    \geq
    \frac{c_1 \omega}{2} \sum_{k \in \mathcal S} \Vert x^{k+1} - x^k \Vert
    =
    \frac{c_1 \omega}{2} \sum_{k = 0}^\infty \Vert x^{k+1} - x^k \Vert ,
    \end{multline*}
    where the last inequality is due to \eqref{eq:d_definition} and \eqref{eq_pred}.
    Thus $\{x^k\}$ is a Cauchy sequence, and thus converges, contradicting the assumption that $\{x^k\}$ is unbounded.
\end{proof}

\begin{theorem}\label{theorem:mu_pred_3}
    Suppose \cref{assumption_general,ass:Bk_bounded} are satisfied.
    Then $\liminf_{k \in \mathcal K} \mu_k \Vert d^k \Vert = 0$ holds.
\end{theorem}
\begin{proof}
    If $\{x^k\}$ is unbounded, the result follows from \cref{lemma:xk_unbounded}.
    Hence, assume that $ \{ x^k \} $ is bounded for the remaining part of the proof.
    If $ \{ \mu_k \} $ is bounded, \cref{alg:RPQN} performs infinitely many successful iterations and the result follows from \cref{theorem:global_convergence_cvx}.
    It remains to assume that $ \{ \mu_k \} $ is unbounded.
    Then, there are infinitely many unsuccessful iterations and we can find an index set $\mathcal{L}^\prime\subset \mathcal{K}$ such that $\{ \mu_k \}_{\mathcal{L}^\prime} \to \infty $.
    Eventually, for $k\in \mathcal{L}^\prime$ large enough, it holds that $k-1 \in \mathcal{U}$, otherwise $\mu_k \leq \mumax$ by the update rule at \lineref{algo_line_K_commands_2} of \cref{alg:RPQN}.
    Without loss of generality, we can assume that $k-1\in\mathcal{U}$ for all $k\in\mathcal{L}^\prime$.
    Now we define $\mathcal{L} \coloneqq \{ k-1 \,|\, k\in\mathcal{L}^\prime \}$, for which $\{x^k\}_{\mathcal{L}}$ remains bounded and $\{\mu_k\}_{\mathcal{L}}\to\infty$.
    Therefore, \cref{lemma:mu_pred_2} yields $\lim_{k \in \mathcal{L}} \mu_k \| d^k \| = 0$, concluding the proof.
\end{proof}

Convergence in terms of the distance of $\partial F(\hat{x}^k)$ from zero can be established when the smooth term $f$ has additional regularity, without assuming existence of accumulation points.
\begin{theorem}\label{thm:liminf_stationarity}
    Suppose that \cref{assumption_general,ass:Bk_bounded} hold.
    If $\nabla f$ is uniformly continuous, then it holds that 
    \[
    \liminf\limits_{k \in \mathcal K}\dist(0,\partial F(\hat x^k)) = 0.
    \]
\end{theorem}
\begin{proof}
    By \cref{lemma:q_hat}, $\mathcal K$ is an infinite set.
    For every iteration $k \in \mathcal K$ it holds by \eqref{eq:localprox} that 
    \begin{multline*}
    	0 \in \partial\hat q_k(\hat x^k) = \nabla f(x^k) + G_k d^k + \partial \varphi(\hat x^k)
    	\\
    	\iff\quad
    	\nabla f(\hat x^k) - \nabla f(x^k) - G_k d^k \in \nabla f(\hat x^k) + \partial \varphi(\hat x^k) = \partial F(\hat x^k).
    \end{multline*}
    Therefore we get
    \[
    \dist(0,\partial F(\hat x^k)) \leq \Vert \nabla f(\hat x^k) - \nabla f(x^k) - G_k d^k \Vert \leq \Vert  \nabla f(\hat x^k) - \nabla f(x^k) \Vert + (\Vert B_k \Vert + \mu_k) \Vert d^k \Vert.
    \]
    According to \cref{theorem:mu_pred_3} it holds that $\liminf_{k \in \mathcal K} \mu_k \Vert d^k \Vert = 0$.
    Together with the uniform continuity of $\nabla f$ and the boundedness of $\{B_k\}$, taking the $\liminf_{k \in \mathcal K}$ on both sides yields the result thanks to \cref{cor:liminf_dk}.
\end{proof}

\subsection{Stationarity of accumulation points}

Under one additional \emph{local} regularity assumption, we can derive a stronger convergence statement.
In particular, this allows us to prove local boundedness of the regularization parameter and, eventually, stationarity of every accumulation point of the generated sequence.
See \cite{kanzow_mehlitz_2022,demarchi2022proximal,demarchi_2023} for related proximal-gradient analyses.

\begin{assumption}\label{ass:loc_lipschitz}
	$\nabla f$ is locally Lipschitz continuous on an open set containing $\dom\varphi$, i.e., for all $x \in \dom\varphi$ there exist $L, \varepsilon > 0$ such that
	\[
	\Vert \nabla f(y) - \nabla f(z) \Vert \leq L \Vert y - z \Vert
	\]
	for all $y,z \in \ball_{\varepsilon}(x)$.
\end{assumption}

With this additional assumption, the following result shows that the regularization parameter remains bounded in the vicinity of accumulation points.

\begin{lemma}\label{lemma:mu_bounded}
    Suppose that \cref{assumption_general,ass:Bk_bounded,ass:loc_lipschitz} hold.
    Then, for every accumulation point $x^*$ of $\{x^k\}$, there exist constants $\varepsilon > 0$ and $\mubar > 0$ such that $\mu_k \leq \mubar$ for all $k \in \N_0$ with $x^k \in \ball_{\varepsilon}(x^*)$.
\end{lemma}
\begin{proof}
    Consider an accumulation point $x^*$ of $\{x^k\}$.
    Assume, by contradiction, that for all $\varepsilon > 0$ and $\mubar > 0$ there exists $k \in \N_0$ with $x^k \in \ball_{\varepsilon}(x^*)$ and $\mu_k > \mubar$.
    Then, owing to \cref{ass:loc_lipschitz}, there exists a local Lipschitz constant $L^*$ of $\nabla f$ in $\ball_{\varepsilon}(x^*)$.
    Moreover, there exists a subset $\mathcal L \subset \N_0$ with $\{x^k\}_{\mathcal L} \to x^*$ and $\{\mu_k\}_{\mathcal L} \to \infty$.
    For all $k \in \mathcal L$ with $\mu_k > \mumax$, $k-1$ was unsuccessful with $\mu_{k-1} = \sigma_2^{-1} \mu_k$ and $x^{k-1} = x^k$ by \lineref{algo_line_K_commands_2} of \cref{alg:RPQN}.
    This shows that we can further assume, without loss of generality, that $\mathcal L \subset \mathcal U$.
    \cref{lemma:q_hat} together with  \cref{ass:Bk_bounded} assures that we can assume, without loss of generality, that $\mathcal L \subset \mathcal K$.
    Furthermore, \cref{lemma:q_hat} also implies $\lim_{k \in \mathcal L} \Vert d^k \Vert = 0$. 
    By Taylor's theorem there exists $\xi^k$ on the line segment between $x^k$ and $\hat x^k$ such that $f(\hat x^k) = f(x^k) + \nabla f(\xi^k)^\top d^k$.
    Consequently, for $k \in \mathcal L$ sufficiently large it holds that
    \begin{align*}
        \pred_k - \ared_k
        ={}&
        f(\hat x^k) - f(x^k) - \innerprod{\nabla f(x^k)}{d^k} - \frac{1}{2} \|d^k\|_{B_k}^2
        ={}
        \innerprod{\nabla f(\xi^k) - \nabla f(x^k)}{d^k} - \frac{1}{2} \|d^k\|_{B_k}^2
        \\
        \leq{}&
        \Vert \nabla f(\xi^k) - \nabla f(x^k) \Vert \Vert d^k \Vert + \frac{1}{2} \Vert B_k \Vert \Vert d^k \Vert^2
        \leq{}
        \left( L^* + \frac{1}{2} \Vert B_k \Vert \right) \Vert d^k \Vert^2,
    \end{align*}
    where we used the Cauchy-Schwarz inequality in the first, and the Lipschitz constant $L^*$ in the second inequality.
    Hence, together with \cref{lemma:pred} this implies
    \begin{equation*}
        \ared_k - c_1 \pred_k
        =
        (1-c_1) \pred_k - (\pred_k - \ared_k)
        \geq
        \frac{1}{2} \big( (1-c_1)\mu_k - 2L^* - \Vert B_k \Vert \big) \Vert d^k \Vert^2.
    \end{equation*}
    By \cref{ass:Bk_bounded} this shows that, for $k \in \mathcal L$ sufficiently large, it holds that $\rho_k \geq c_1$, which is a contradiction to $k \in \mathcal U$.
\end{proof}

Recall that $\varphi$, and therefore $F$, is only lower semicontinuous and need not be continuous.
Hence, convergence of a subsequence $\{x^k\}$ does not automatically imply convergence of the associated function values $\{F(x^k)\}$ to $F(x^*)$.
Even so, the next result establishes a stronger conclusion, made possible by the specific structure of the proximal-type iteration.

\begin{lemma}\label{Prop:Fxk-converges}
    Suppose that \cref{assumption_general,ass:Bk_bounded,ass:loc_lipschitz} hold.
    Let $x^*$ be an accumulation point of $\{x^k\}$.
    Assume that either $\varphi$ is continuous on $\dom \varphi$, or the proximal subproblem~\eqref{eq_subproblem} is solved globally, or both.
    Then the entire sequence $ \{ F(x^k) \} $ converges to $ F(x^*)$.
\end{lemma}
\begin{proof}
    By \cref{lemma:basic_properties}\ref{lemma:basic_F}, the sequence $ \{ F(x^k) \} $ is monotonically decreasing and, by \cref{assumption_general}, bounded from below.
    Hence, $ \{ F(x^k) \} $ is convergent to some finite number $ F_* > - \infty $.
    There exists a subset $\mathcal L \subset \N_0$ with $\{x^k\}_{\mathcal L} \to x^*$ and, without loss of generality, we can assume that $\mathcal L \subset \mathcal S$.
    Due to the lower semicontinuity of $ F $, we get 
    \begin{equation*}
        F_* = \lim_{k \to \infty} F(x^k) = \liminf_{k \in \mathcal L} F(x^k) \geq F(x^*).
    \end{equation*}
    Thus it suffices to show that $F_* \leq F(x^*)$, which we now verify under either of the two assumptions.
    
    \emph{When $\varphi$ is continuous on $\dom \varphi$.}
    Recall that, for each $ k \in \mathcal L$, $x^{k+1} = \hat x^k$ solves the subproblem~\eqref{eq_subproblem} in the sense of \eqref{eq:localprox}.
    In particular, we therefore obtain
    \begin{equation}\label{eq:conv-inequality}
    	\hat q_k (x^{k+1})
    	={}
    	f(x^k) + \innerprod{\nabla f(x^k)}{x^{k+1}-x^k} + 
    	\frac{1}{2} \|x^{k+1}-x^k\|_{G_k}^2 +
    	\varphi (x^{k+1})
    	\leq
    	\hat{q}_k(x^k)
    	=
    	F(x^k).
    \end{equation}
    From \cref{theorem:global_convergence_cvx} we know that $\{d^k\}_{\mathcal L} \to 0$.
    Taking into account that $ d^k = x^{k+1} - x^k $ for $k \in \mathcal S$, this implies $ \{ x^{k+1} \}_{\mathcal L} \to x^* $.
    Furthermore, \cref{lemma:mu_bounded} together with the boundedness of $ \{ B_k \} $ then yields the boundedness of the subsequence $ \{ G_k \}_{k \in \mathcal L} $. 
    These considerations yield
    \begin{multline*}
    	F_*
    	={}
    	\lim_{k \to \infty} F(x^k)
    	=
    	\lim_{k \in \mathcal L} F(x^{k+1})
    	={}
    	\lim_{k \in \mathcal L} \big[ f(x^{k+1}) + 
    	\varphi (x^{k+1}) \big]
    	=
    	\lim_{k \in \mathcal L} \big[ f(x^{k}) + 	\varphi (x^{k+1}) \big]
    	\\
    	={}
    	\lim_{k \in \mathcal L} \Big[
    	f(x^k) + \innerprod{\nabla f(x^k)}{x^{k+1}-x^k} + 
    	\frac{1}{2} \| x^{k+1}-x^k \|_{G_k}^2 +
    	\varphi (x^{k+1}) \Big]
    	\leq{}
    	\lim_{k \in \mathcal L} F(x^k)
    	=
    	F(x^*),
    \end{multline*}
    where the inequality is due to \eqref{eq:conv-inequality}
    and the limit due to the assumed continuity of $\varphi$ on its domain.
    
    \emph{When subproblem~\eqref{eq_subproblem} is solved globally.}
    By global optimality of $x^{k+1}$ for $\hat{q}_k$ we find
    \begin{multline}\label{eq:conv-inequality-glob}
        \hat q_k (x^{k+1})
        ={}
        f(x^k) + \innerprod{\nabla f(x^k)}{x^{k+1}-x^k} + 
        \frac{1}{2} \|x^{k+1}-x^k\|_{G_k}^2 +
        \varphi (x^{k+1})
        \\
        \leq
        \hat{q}_k(x^*)
        =
        f(x^k) + \innerprod{\nabla f(x^k)}{x^*-x^k} + 
        \frac{1}{2} \|x^*-x^k\|_{G_k}^2 +
        \varphi (x^*).
    \end{multline}
    Recalling the arguments from the previous case, taking the limit yields
    \begin{multline*}
    	F_*
    	={}
    	\lim_{k \in \mathcal L} \Big[
    	f(x^k) + \innerprod{\nabla f(x^k)}{x^{k+1}-x^k} + 
    	\frac{1}{2} \| x^{k+1}-x^k \|_{G_k}^2 +
    	\varphi (x^{k+1}) \Big]
    	\\
    	\leq{}
    	\lim_{k \in \mathcal L} \Big[
    	f(x^k) + \innerprod{\nabla f(x^k)}{x^*-x^k} + 
    	\frac{1}{2} \|x^*-x^k\|_{G_k}^2 +
    	\varphi (x^*) \Big]
    	=
    	\lim_{k \in \mathcal L} f(x^k) + \varphi (x^*)
    	=
    	F(x^*),
    \end{multline*}
    where the inequality is due to \eqref{eq:conv-inequality-glob}.
    
    Altogether, this verifies $F_* \leq F(x^*)$ under both circumstances and completes the proof.
\end{proof}

\begin{remark}[Local proximity operator and non-attentive convergence]\label{rem:localprox_attentive}
	In general, when solving subproblem~\eqref{eq_subproblem} in the local sense of \eqref{eq:localprox}, the convergence of $F(x^k)$ to $F(x^*)$ can fail if $\varphi$ is not continuous on $\dom \varphi$.
	We provide a one-dimensional counterexample with convex $f$.
		
	Take $n=1$, $f(x)\coloneqq \tfrac{1}{2}x^2$, $\varphi(x) \coloneqq 0$ at $x=0$ and $1$ elsewhere (lsc, bounded below), $B_k=1$, $\mu_k=1$ (so $G_k=2$).
	Therefore,
	$\hat{q}_k(x) \coloneqq \tfrac{1}{4}(x^k)^2 + \left(x-\tfrac{1}{2}x^k\right)^2 + \varphi(x)$
	and, at any $x^k \neq 0$, the local proximity oracle may return $\hat{x}^k = \tfrac{1}{2}x^k$, which is stationary for $\hat{q}_k$ and yields a value lower than $\hat{q}_k(x^k)$.
	Then $\ared_k = \pred_k = \tfrac{3}{8} (x^k)^2$, so every step is accepted.
	Hence, starting from any $x^0 \neq 0$, \cref{alg:RPQN} can generate the sequence $\{x^k\}$ according to $x^{k+1} \coloneqq \tfrac{1}{2} x^k \neq 0$.
	The iterates converge to $x^*=0$, but $F(x^k)\to1 \neq F(x^*)=0$.
\end{remark}

We have another technical lemma before turning to \cref{theorem:global_convergence_1}.

\begin{lemma}
	\label{lemma:acc_points}
	Suppose that \cref{assumption_general,ass:Bk_bounded} hold.
	Then for every accumulation point $x^*$ of $\{x^k\}$ there exists a subset $\mathcal L \subset \mathcal K$ with $\{x^k\}_{\mathcal L} \to x^*$ and $\lim_{k \in \mathcal L} \mu_k \Vert d^k \Vert = 0$. 
\end{lemma}
\begin{proof}
	If \cref{alg:RPQN} performs only finitely many successful iterations, then the result follows immediately from \cref{lemma:q_hat,lemma:mu_pred_2}.
	Now suppose that $\mathcal S$ is an infinite set. Since $x^*$ is an accumulation point of $\{x^k\}$, there exists $\mathcal L \subset \N_0$ with $\{x^k\}_{\mathcal L} \to x^*$.
	We construct the set $\mathcal L' \subset \mathcal S$ by replacing unsuccessful iterations $k \in \mathcal L$ with the last successful iteration before $k$.
	As $\mathcal S$ is an infinite set, the same holds for $\mathcal L'$. From \cref{theorem:global_convergence_cvx} it follows that $\{x^k\}_{\mathcal L'} \to x^*$.
	If $\{\mu_k\}_{\mathcal L'}$ is bounded, the result follows immediately with \cref{theorem:global_convergence_cvx}. If $\{\mu_k\}_{\mathcal L'}$ is unbounded, passing to a subsequence if necessary, we can assume that $k-1 \in \mathcal U$ for all $k \in \mathcal L'$, as for $k-1 \in \mathcal S$ it would be $\mu_k \leq \mumax$ by \lineref{algo_line_K_commands_2}.
	Now consider the set $\mathcal U' \coloneqq \{k-1 \, \vert \, k \in \mathcal L'\}$.
	Clearly, for $\mathcal U'$ it also holds that $\{x^k\}_{\mathcal U'} \to x^*$ and $\{\mu_k\}_{\mathcal U'} \to \infty$.
	By \cref{lemma:q_hat}, again passing to a subsequence if necessary, we have $\mathcal U' \subset \mathcal U$.
	Then the result follows from \cref{lemma:mu_pred_2}.
\end{proof}

\begin{theorem}\label{theorem:global_convergence_1}
    Suppose that \cref{assumption_general,ass:Bk_bounded,ass:loc_lipschitz} hold.
    Let $x^*$ be an accumulation point of $\{x^k\}$, namely $x^k \to_{\mathcal{L}} x^*$ for some subset $\mathcal{L} \subset \mathcal{K}$.
    If $F(x^k) \to_{\mathcal{L}} F(x^*)$, then $x^*$ is a stationary point of \eqref{problem}.
\end{theorem}
\begin{proof}
	From \cref{lemma:acc_points} we know that there exists a subset $\mathcal L \subset \mathcal K$ with $\{x^k\}_{\mathcal L} \to x^*$ and $\lim_{k \in \mathcal L} \mu_k \Vert d^k \Vert = \lim_{k \in \mathcal L} \Vert d^k \Vert = 0$.
	From the assumption $F(x^k) \to_{\mathcal{L}} F(x^*)$ we obtain
	\begin{equation*}
		\lim_{k \in \mathcal L} \varphi(x^k)
		=
		\lim_{k \in \mathcal L} F(x^k) - f(x^k)
		=
		F(x^*) - f(x^*)
		=
		\varphi (x^*).
	\end{equation*}
	Optimality of $\hat x^k$ with respect to \eqref{eq_subproblem}, according to \eqref{eq:localprox}, combined with the Cauchy-Schwarz inequality gives
	\begin{multline*}
		F(x^k)
		=
		\hat q_k(x^k)
		\geq
		\hat q_k(\hat x^k)
		=
		f(x^k) + \innerprod{\nabla f(x^k)}{d^k} + \frac{1}{2} \|d^k\|_{G_k}^2 + \varphi(\hat x^k) 
		\\
		\geq
		f(x^{k}) - \Vert \nabla f(x^{k}) \Vert \Vert d^k \Vert + \frac{\mu_k - \Vert B_{k} \Vert}{2} \Vert d^k \Vert^2 + \varphi(\hat x^k),
	\end{multline*}
	hence
	\[
	\varphi(x^{k}) \geq \varphi(\hat x^k) - \Vert \nabla f(x^{k}) \Vert \Vert d^k \Vert + \frac{\mu_k - \Vert B_{k} \Vert}{2} \Vert d^k \Vert^2
	\]
	for all $k \in \mathcal L$. Taking the $\limsup$ along $\mathcal L$ gives
	\[
	\liminf_{k \in \mathcal L} \varphi(\hat x^k) \geq \varphi(x^*) \geq \limsup_{k \in \mathcal L} \varphi(\hat x^k),
	\]
	where the first inequality is due to lower semicontinuity of $\varphi$ and $\{\hat x^k\}_{\mathcal L} \to x^*$.
	Therefore, $\lim_{k \in \mathcal L} \varphi(\hat x^k) = \varphi(x^*)$.
	Finally, since $\hat x^k \in \argmin \hat q_k$ in the sense of \eqref{eq:localprox}, the optimality condition results in
	\[
	0 \in \partial \hat q_k (\hat x^k) = \nabla f(x^k) + (B_k+\mu_k I) d^k + \partial \varphi(\hat x^k)
	\]
	for all $k \in \mathcal L$.
	Using the aforementioned limits along with the outer semicontinuity (with respect to $\varphi$-attentive convergence) of the limiting subdifferential, taking the limit along $\mathcal L$ gives
	\[
	0 \in \nabla f(x^*) + \partial\varphi(x^*) = \partial F(x^*),
	\]
	concluding the proof.
\end{proof}

\subsection{Convergence under the KL property}\label{Sec:Convergence}

The first aim of this section is to prove convergence
of the entire sequence $\{x^k\}$ generated by \cref{alg:RPQN} under the following assumptions (in addition to \cref{assumption_general,ass:Bk_bounded,ass:loc_lipschitz}).
\begin{assumption}\label{assumption_loc_prox_cvx}
    Consider problem \eqref{problem} and the iterates $\{x^k\}$ generated by \cref{alg:RPQN}.
    \begin{enumerate}[label=(\alph*)]
        \item The set $\solutionset$ of stationary points of \eqref{problem} is nonempty and there exists an accumulation point $x^* \in \solutionset$ of $\{x^k\}_{\mathcal S}$ with $F(x^k) \to_{\mathcal{S}} F(x^*)$. \label{ass_existence_cvx}
        \item $F$ has the KL property at $x^*$ with desingularization function $\chi$, constant $\eta > 0$ and neighborhood $U$ of $x^*$, see \cref{def_KL}. \label{ass_KL}
    \end{enumerate}
\end{assumption}
Alongside error-bound conditions, the KL property is now widely used to show convergence and rate-of-convergence results for proximal methods.
A major milestone has been \cite{bolte_2016}, where for convex, lsc functions the equivalence between the KL property of exponent $1/2$ and an error bound is established.
In recent years, the KL property has become the default assumption in modern convergence analysis.
One of the reasons for that is that many of the objective functions in practical applications are KL with known exponents \cite{bolte_2007}.

In the following, we focus on an arbitrary but fixed accumulation point $x^*$ and let $L^* > 0$ be the local Lipschitz constant of $\nabla f$ in $\ball_{\varepsilon_0}(x^*)$ for some radius $\varepsilon_0 > 0$.
Then, by \cref{lemma:mu_bounded}, there exists $\mubar > 0$ and $0 < \varepsilon_1 < \varepsilon_0$ such that $\mu_k \leq \mubar$ for all $k \in \N_0$ with $x^k \in \ball_{\varepsilon_1}(x^*)$.

\begin{lemma}\label{lemma_alpha}
    Suppose that \cref{assumption_general,ass:Bk_bounded,ass:loc_lipschitz,assumption_loc_prox_cvx} hold.
    Define
    \[
    \beta \coloneqq \frac{c_1 \mumin}{2(M_B+\mubar+L^*)}
    \quad\text{and}\quad
    \alpha_k \coloneqq \Vert x^{k} - x^* \Vert + \sqrt{\frac{8}{c_1\mumin}\left( F(x^{k})-F(x^*)\right)} + \frac{1}{\beta} \chi \left( F(x^{k}) - F(x^*) \right).
    \] 
    Then, it holds that $\liminf\limits_{k \in \mathcal S} \alpha_k = 0$.
\end{lemma}
\begin{proof}
    According to \cref{assumption_loc_prox_cvx}\ref{ass_existence_cvx} there exists a subsequence of $\{x^k\}_{\mathcal S}$ converging to $x^*$.
    Along such a subsequence, each summand in $\alpha_k$ tends to zero:
    the first by convergence of $x^k$ to $x^*$,
    the second by convergence of $F(x^k)$ to $F(x^*)$ (recall \cref{Prop:Fxk-converges}),
    and the third because the desingularization function $\chi$ is continuous at the origin.
    Hence $\alpha_k \to 0$ along that subsequence.
\end{proof}
The following intermediate result gives an upper bound on a stationarity measure in terms of the distance between consecutive iterates.
In particular, this motivates to use this distance as a termination criterion.
The proof of this result is given in \cref{sec:omitted_proofs}.
\begin{lemma}\label{lemma_dist}
    Suppose that \cref{assumption_general,ass:Bk_bounded,ass:loc_lipschitz,assumption_loc_prox_cvx} hold.
    For every iteration $k \in \mathcal S$ with $x^k \in \ball_{\varepsilon_1}(x^*)$ and $\Vert d^k \Vert \leq \varepsilon_0 - \varepsilon_1$ it holds that
    \[
    \dist\left(0,\partial F(x^{k+1})\right) \leq \left( M_B + \mubar + L^* \right) \Vert x^{k+1}-x^k \Vert.
    \]
\end{lemma}

The following result guarantees convergence of the entire sequence $ \{ x^k \} $ under \cref{assumption_general,ass:Bk_bounded,ass:loc_lipschitz,assumption_loc_prox_cvx}.
Note that, in general, these assumptions do not necessarily imply the local uniqueness of a solution.
A key departure from the proof of \cite[Thm 4.5]{jia_kanzow_mehlitz_2023} is that our induction runs only over iterations in $\mathcal S$ from $k_0$ onward, rather than over all $k \geq k_0$.
This suffices because $x^{k+1}=x^k$ holds for all nonsuccessful iterations.
The proof is given in \cref{sec:omitted_proofs}.

\begin{theorem}\label{theorem_final}
    Suppose that \cref{assumption_general,ass:Bk_bounded,ass:loc_lipschitz,assumption_loc_prox_cvx} hold. 
    Then the entire sequence $\{x^k\}$ converges to $x^*$.
\end{theorem}

We now turn from convergence itself to quantitative rates, assuming a KL desingularization of the form $\chi(t)=c\,t^{1-\theta}$.
The analysis naturally applies to the successful iterations, because those are the indices where objective decrease is explicitly controlled (and because the iterates remain constant at unsuccessful iterations).
A direct extension to the full sequence is generally unavailable: unsuccessful iterations may still occur at arbitrarily large indices and interrupt any global Q-type rate statement for all indices.
The first statement of the next theorem resolves this mechanism by showing that such interruptions can only happen in uniformly bounded blocks.
The proof is left to \cref{sec:omitted_proofs}.

\begin{theorem}\label{thm:kl-successful-recursion}
    Suppose that \cref{assumption_general,ass:Bk_bounded,ass:loc_lipschitz,assumption_loc_prox_cvx} hold and that $F$ has the KL property at $x^*$ with desingularization function $\chi(t)=c\,t^{1-\theta}$ for some $c>0$ and $\theta\in(0,1)$.
    Let $s_0,s_1,s_2,\dots$ denote the successful iterations, with $s_0\coloneqq k_0$.
    Then the sequence $\{s_{j+1}-s_j\}$ is bounded and the following statements hold:
    \begin{enumerate}[(a)]
        \item If $\theta\in(0,\tfrac12)$, then $\{F(x^{s_k})\}$ converges
        Q-superlinearly to $F(x^*)$ with rate $\frac{1}{2\theta}$, and
        $\{x^{s_k}\}$ converges R-superlinearly to $x^*$ with the same rate.
        \item If $\theta=\tfrac12$, then $\{F(x^{s_k})\}$ converges Q-linearly
        to $F(x^*)$, and $\{x^{s_k}\}$ converges R-linearly to $x^*$.
        \item If $\theta\in(\tfrac12,1)$, then there exist constants
        $C_1,C_2>0$ such that for all sufficiently large $k$,
        \begin{align*}
            F(x^{s_k}) - F(x^*) \leq{}& C_1\, k^{-\frac{1}{2\theta-1}},
            &
            \|x^{s_k}-x^*\|     \leq{}& C_2\, k^{-\frac{1-\theta}{2\theta-1}}.
        \end{align*}
    \end{enumerate}
\end{theorem}

\section{Solution of the proximal subproblems}\label{sec_subproblems}

Subproblems at \lineref{algo_line_subproblem} turn out to be local proximal evaluations, namely $\hat x^k \in \localprox_{f,\varphi}^{G_k}\left(x^k\right)$.
How the local proximal operator can be efficiently evaluated for limited-memory quasi-Newton approximations is the topic of this section.

We begin by recalling from \cref{lem:scaled_prox_relation} that, whenever $G_k\in\Smat_{++}^n$ (which holds for sufficiently large $\mu_k>0$, owing to \cref{ass:Bk_bounded}), then
\[
	\prox_\varphi^{G_k} \left( x^k - G_k^{-1}\nabla f(x^k) \right) \subseteq \localprox_{f,\varphi}^{G_k}\left(x^k\right).
\]
Therefore, subproblems arising at \lineref{algo_line_subproblem} could be tackled in essentially the same way as in \cite[Section~3.3]{lechner_2022},
which exploits compact representations of quasi-Newton approximations to efficiently solve the scaled proximal subproblem \eqref{eq_subproblem}.
However, their approach is valid for convex regularizers $\varphi$.
We show that it is possible to use it, and how, for nonconvex $\varphi$ too.

In each iteration of \cref{alg:RPQN} we must compute a solution of subproblem \eqref{eq_subproblem}, at least in the sense of \eqref{eq:localprox}, where $G_k \coloneqq B_k + \mu_k I\in\Smat^n$ and $B_k \approx \nabla^2 f(x^k)$.
If $B_k$ is obtained by a limited-memory BFGS, SR1, or Kleinmichel update, then it admits a compact representation
\begin{equation}
	\label{eq_compact}
	B_k = B_{k,0} + Q_k M_k^{-1} Q_k^\top,
\end{equation}
see \cite[Section~3.3.1]{lechner_2022} and \cref{sec:kleinmichel}.
Following \cite[Section~3.3.2]{lechner_2022}, this representation can be rewritten in the form
\[
B_k = B_{k,0} + U_{k,1} U_{k,1}^\top - U_{k,2} U_{k,2}^\top,
\]
with matrices $U_{k,i} \in \R^{n\times r_i}$ of small rank $r_i>0$ $(i=1,2)$.
Consequently,
\begin{equation}
	\label{eq_Gk}
	G_k = B_k + \mu_k I = B_{k,0}+\mu_kI + U_{k,1} U_{k,1}^\top - U_{k,2} U_{k,2}^\top,
\end{equation}
and, following \cite[Theorem~3.21]{lechner_2022}, this representation can be used to compute efficiently the proximal point $\prox_\varphi^{G_k} \left( x^k - G_k^{-1}\nabla f(x^k) \right)$ provided that $G_k\in\Smat_{++}^n$ and $ \varphi $ is convex.
Since this convexity assumption may not hold in our case, we add some corresponding discussion in the following subsection.

\subsection{Reduction to a low-dimensional subproblem}

Let $x\in\R^n$ and, for simplicity, $H=H_0+U U^\top$ with $H_0\in\Smat_{++}^n$ and $U\in\R^{n\times r}$, so that $ H $ results from $ H_0 $ by a (at most) rank $ r $ modification of $ H_0 $.
Then consider the two functions
\begin{equation*}
	g(y) \coloneqq{} \varphi(y)+\tfrac12\|y-x\|_H^2
	\qquad \text{and} \qquad
	\Phi(y;w) \coloneqq{} \varphi(y)+\frac12\|y-w\|_{H_0}^2
\end{equation*}
for some given $ w \in \R^n $.
The underlying idea of the approach by Becker, Fadili \& Ochs \cite{becker_fadili_ochs_2019}, further exploited in Lechner \cite{lechner_2022}, is that the proximal point in the sense of computing a minimum of $ \Phi $ (with respect to $ y $) is a simple task (in general, $ H_0 $ is a positive multiple of the identity matrix), and that the particular structure of the matrix $ H $ can then be used to obtain a minimum of $ g $ with only a small amount of extra work.
This is particularly promising since, in contrast to $ H_0 $, the matrix $ H $ might contain second-order information of the smooth part $f$ of the objective function $F$, so that our local model function provides a better approximation.
Unfortunately, however, the techniques used in \cite{becker_fadili_ochs_2019,lechner_2022} assume $ \varphi $ to be convex, whereas we also allow nonconvex functions $ \varphi $ in our setting.
On the other hand, the subsequent discussion shows that essentially the same approach as in \cite{becker_fadili_ochs_2019,lechner_2022} can still be used to compute a vector $ y^* $ such that $ 0 \in \partial g(y^*) $ and 
$ g(y^*) \leq g(x) $ (i.e., \eqref{eq:localprox} holds), while accessing $ \varphi $ only through the proximal operator $ \prox_\varphi^{H_0} $.

To be more precise, let us introduce
\begin{align*}
	w(\alpha) \coloneqq{}& x-H_0^{-1}U\alpha ,&
	J(w) \coloneqq{}& \{y:\ 0\in\partial \varphi(y)+H_0(y-w)\}, \\
	\Lambda(y,\alpha) \coloneqq{}& U^\top(x-y)+\alpha ,&
	M \coloneqq{}& U^\top H_0^{-1}U \in\Smat_+^r.
\end{align*}
for any $ y \in \R^n $ and $ \alpha \in \R^r $.
Using this notation, we have the obvious implications
\begin{equation*}
	y \in \prox_\varphi^H\big( x \big)
	\ \implies\
	0\in\partial g(y)
	\qquad \text{ and } \qquad 
	y\in \prox_\varphi^{H_0}\big(w(\alpha)\big)
	\ \implies \ y\in J\big(w(\alpha)\big) .
\end{equation*}
The following result shows that the computation of an $ n $-dimensional stationary point of $ g $ can essentially be reduced to solving an $ r $-dimensional problem.
\begin{proposition}\label{prop:prox_station}
	Let $y\in\R^n$.
	Then, $ 0 \in \partial g(y) $ if and only if there exists $ \alpha \in \R^r $ satisfying $ \Lambda (y, \alpha) = 0 $ and $ y \in J(w (\alpha ))$.
\end{proposition}
\begin{proof}
	Since $y\mapsto\tfrac12\|x-y\|_H^2$ is $C^1$, it follows from the sum rule that $\partial g(y)=\partial \varphi(y)+H(y-x)$.
	With $\alpha=U^\top(y-x)$ from the definition of $\Lambda(y,\alpha)$, we have
	\[
	H(y-x)=H_0(y-x)+U\alpha=H_0(y-x+H_0^{-1}U\alpha)=H_0\big(y-w(\alpha)\big),
	\]
	so
	\[
	0\in\partial g(y)
	\iff
	0\in\partial \varphi(y)+H_0(y-w(\alpha))
	\iff
	y\in J(w(\alpha)).
	\]
	This completes the proof.
\end{proof}
The following two useful identities are proved in \cref{sec:omitted_proofs}.
For all $y\in\R^n$ and $\alpha\in\R^r$,
\begin{align}
	g(y) ={}& \Phi(y;w(\alpha))+\tfrac12\|\Lambda(y,\alpha)\|^2-\tfrac12\|\alpha\|^2-\tfrac12\alpha^\top M\alpha,\label{eq:ident1}\\
	g(y) - g(x) ={}& \Phi(y;w(\alpha))-\Phi(x;w(\alpha))+\tfrac12\|\Lambda(y,\alpha)\|^2-\tfrac12\|\alpha\|^2.\label{eq:ident2}
\end{align}
We next discuss the relation between the local and global minima.
\begin{proposition}\label{prop:localcert}
	Let $\alpha\in\R^r$ and $y\in\R^n$ satisfy $\Lambda(y,\alpha)=0$.
	If $y$ is a local minimizer (resp.\ global minimizer) of $\Phi(\cdot;w(\alpha))$,
	then $y$ is a local (resp.\ global) minimizer of $g$,
	and $g(y)=\Phi(y;w(\alpha))-\tfrac12\|\alpha\|^2-\tfrac12\alpha^\top M\alpha$.
\end{proposition}
\begin{proof}
	Let $c_\alpha \coloneqq \tfrac12\|\alpha\|^2+\tfrac12\alpha^\top M\alpha$.
	Then, by \eqref{eq:ident1}, $g(y) = \Phi(y;w(\alpha))-c_\alpha$ and $g(y')\geq \Phi(y';w(\alpha))-c_\alpha$ for all $y'\in\R^n$, with equality at $y'=y$.
	For $y'$ near $y$ (in the local case), $\Phi(y';w(\alpha))\geq\Phi(y;w(\alpha))$
	(locally), hence $g(y')\geq\Phi(y;w(\alpha))-c_\alpha=g(y)$ (locally).
\end{proof}

In general, the converse of \cref{prop:localcert} fails.
\begin{remark}
	If $\bar{y}$ is a local minimizer of $g$ (with associated $\bar{\alpha}$), it is not necessarily a local minimizer of $\Phi(\cdot;w(\bar{\alpha}))$.
	Such $\bar{y}$ are still captured by \cref{prop:prox_station} (as saddle points of $\Phi$), but not by \cref{prop:localcert}.
	We provide a one-dimensional counterexample with smooth $\varphi$.
	
	Take $n=r=1$, $H_0=1$, $U=1$ (so $H=2$), $x=0$, and $\varphi(y)=\tfrac14y^4-\tfrac34y^2$.
	Therefore,
	\begin{itemize}
		\item $g(y)=\varphi(y)+y^2=\tfrac14y^4+\tfrac14y^2$ is strictly convex with unique stationary point and global minimizer $\bar{y} = 0$, with associated $\bar{\alpha}=0$ and $w(\bar{\alpha}) = 0$.
		\item $\Phi(y;w(\bar{\alpha}))=\varphi(y)+\tfrac{1}{2}y^2=\tfrac14y^4-\tfrac14y^2$ is nonconvex and has a local maximum at $0$; its global minimizers are $\pm \tfrac{1}{\sqrt{2}}$ with value $-\tfrac1{16}$.
	\end{itemize}
	Hence, $\bar{y}=0$ is simultaneously the global minimizer of $g$ and a local maximizer of $\Phi(\cdot;w(\bar{\alpha}))$.
\end{remark}

Recall that we are searching for (at least) a stationary point of $ g $ whose function value is not worse than the previous one.
The following result deals with this problem.

\begin{proposition}\label{prop:prox_descent}
	Let $y\in\R^n$ and set $\alpha = U^\top(y-x)\in\R^r$.
	Then
	\begin{equation*}
		g(y)\leq g(x)
		\quad\iff\quad
		\Phi(y;w(\alpha)) \leq \Phi(x;w(\alpha)) + \tfrac12\|\alpha\|^2.
	\end{equation*}
\end{proposition}
\begin{proof}
	The construction of $\alpha$ implies $\Lambda(y,\alpha)=0$.
	Hence, \eqref{eq:ident2} gives the identity
	\[
	g(y)-g(x)=\Phi(y;w(\alpha))-\Phi(x;w(\alpha))-\tfrac12\|\alpha\|^2.
	\]
	Rearranging yields the desired result.
\end{proof}

The central point of \cref{prop:prox_descent,prop:prox_station} is that they can replace the (difficult) evaluation of the scaled proximal operator $\prox_\varphi^H$ with the (hopefully simpler) evaluation of the proximal operator $\prox_\varphi^{H_0}$, where $H_0$ is often a diagonal matrix or even a multiple of the identity.
A procedure for attempting to compute $\prox_\varphi^H(x)$ is then given in \cref{alg:bfo}, which gives a heuristic implementation of \cite[Thm 3.4]{becker_fadili_ochs_2019} for nonconvex $\varphi$ (and possibly indefinite $G_k\in\Smat^n$); it is exact for convex $\varphi$ and $G_k\in\Smat_{++}^n$.
Whenever the proximal mapping $\prox_\varphi^{H_0}$ is set-valued, a \emph{selection operator} can be adopted to extract an element (e.g., based on minimum-norm or lexicographic order).

\begin{algorithm}[tbh]
	\caption{Searching stationary points for the scaled proximity operator} 
	\label{alg:bfo}
	\begin{algorithmic}[1]
		\Require $x\in\R^n$, $U\in\R^{n\times r}$, $H_0\in\Smat_{++}^n$, $\prox_\varphi^{H_0}\colon\R^n\rightrightarrows\R^n$.
		\Ensure $y^*\in\R^n$ such that $0\in\partial g(y^*)$.
		\State Choose a selection operator $S\colon\R^n\to\R^n$.
		\State Define $Y \coloneqq \alpha \mapsto S(\prox_\varphi^{H_0}(x-H_0^{-1}U\alpha))$.
		\State Define $\Xi\coloneqq \alpha \mapsto U^\top (x-Y(\alpha))+\alpha$.
		\State Find $\alpha^*\in\R^r$ such that $\Xi(\alpha^*)=0$.
		\State \Return $y^* \coloneqq Y(\alpha^*)$.
	\end{algorithmic}
\end{algorithm}

Applied with $H \coloneqq G_k$ from representation~\eqref{eq_Gk} and $y \coloneqq x^k - G_k^{-1}\nabla f(x^k)$, \cref{alg:bfo} uses the oracle $\prox_{\varphi}^{H_0}$ (which can often be computed analytically when $H_0$ is a scaled identity matrix)
to search for $\hat x^k \in \prox_{\varphi}^{G_k}(y)$ in the sense of \cref{def:local_prox_operator}.
The computation is reduced to solving the low-dimensional system $\Xi(\alpha)=0$ in $\R^r$, where $r$ is typically very small compared with $n$.

The mapping $\Xi$ is Newton differentiable for a variety of regularizers $\varphi$, and a generalized derivative of $\Xi$ can be expressed explicitly in terms of a generalized derivative of $\prox_{\varphi}^{H_0}$; see \cite[Proposition~3.22]{lechner_2022} for details.
For many typical choices of $\varphi$, such as the $\ell_1$- and $\ell_2$-norms and group-sparsity penalties, this generalized derivative can be computed analytically.

\section{Limited-memory Kleinmichel matrices}\label{sec:kleinmichel}

This section is devoted to a lesser-known quasi-Newton method, the \textit{Kleinmichel formula}.
This method is a rank-one method that, unlike the standard SR1 formula, guarantees positive definiteness under the same circumstances as the BFGS formula.
The Kleinmichel formula is defined by
\begin{align}
    \label{eq_kleinmichel}
    H_{k+1} &= \gamma_k \left[ H_k + \frac{\left( y^k - \gamma_k H_k d^k \right)\left(y^k-\gamma_k H_k d^k \right)^\top}{\gamma_k \innerprod{ y^k - \gamma_k H_k d^k }{d^k}} \right],
\end{align}
where $\gamma_k > 0$ is a free parameter and $y^k \coloneqq \nabla f(x^{k+1})-\nabla f(x^k)$.
Note that $\gamma_k = 1$ yields exactly the SR1-formula.
For the sake of completeness, we recount this important property, with a simple proof, from the original \cite[Satz 1]{kleinmichel_1981_1} (in German).
\begin{proposition}
    \label{theorem_km_pd}
    Suppose that $H_k$ is positive definite, $\innerprod{y^k}{d^k} > 0$ and $\gamma_k \in \left( 0, \frac{\innerprod{y^k}{d^k}}{\innerprod{d^k}{H_k d^k}}\right)$.
    Then $H_{k+1}$ as defined by the Kleinmichel formula \eqref{eq_kleinmichel} is positive definite.
\end{proposition}
\begin{proof}
    For $\gamma_k > 0$ and positive definite $H_k$, it is clear that $H_{k+1}$ is positive definite if $\innerprod{y^k - \gamma_k H_k d^k}{d^k} > 0$ holds.
    However, from $\gamma_k < \frac{\innerprod{y^k}{d^k}}{\innerprod{d^k}{H_k d^k}}$, it immediately follows that
    \begin{align*}
    	\innerprod{y^k - \gamma_k H_k d^k}{d^k}
        =
        \innerprod{y^k}{d^k} - \gamma_k \innerprod{d^k}{H_k d^k}
        >
        0,
    \end{align*}
    which completes the proof.
\end{proof}
So, even though the Kleinmichel update is a rank-one update, it preserves positive definiteness as long as $\innerprod{y^k}{d^k} > 0$ holds, just like the BFGS update.
The initial matrix $H_0$ is usually chosen to have a simple structure (a positive multiple of the identity matrix, for instance).

In large-scale applications it is not practical to store and update a dense matrix $H_k$ explicitly.
Instead, one usually employs a limited-memory strategy: at iteration $k$ the matrix $H_{k+1}$ is reconstructed from a simple initialization $H_{k,0}$ and only the $m$ most recent quasi-Newton pairs $(d^j,y^j)$, $j = k-m,\dots,k-1$, for some fixed memory size $m \ll k$.
This idea goes back to Nocedal’s fundamental work on L-BFGS \cite{nocedal_1980} and leads to limited-memory quasi-Newton methods with memory parameter~$m$.
In this setting the initialization $H_{k,0}$ may depend on the current iterate $k$ and is not necessarily equal to the original matrix $H_0$.

A crucial ingredient for the efficient numerical use of such methods is the existence of a compact representation of the form
\[
H_k = H_{k,0} + Q_k M_k^{-1} Q_k^\top,
\]
where $H_{k,0} \in \Smat_{++}^n$ is the initialization matrix, $Q_k \in \R^{n\times s}$ has only few columns ($s \ll n$), and $M_k \in \Smat^s$ is nonsingular.
Byrd, Nocedal \& Schnabel \cite[Theorems~2.3 and~5.1]{byrd_nocedal_schnabel_1994} showed that the L-BFGS and L-SR1 matrices admit such a representation.  

In order to derive an analogous representation for the Kleinmichel update, it is convenient to introduce the step and gradient-difference matrices
\begin{equation}\label{eq:kleinmichel_Sk_Yk}
	S_k \coloneqq \bigl[d^0 \; d^1 \; \dots \; d^{k-1}\bigr] \in \R^{n\times k}
	\quad\text{and}\quad
	Y_k \coloneqq \bigl[y^0 \; y^1 \; \dots \; y^{k-1}\bigr] \in \R^{n\times k}.
\end{equation}
The following theorem provides a compact representation of $H_k$ in terms of $H_0$, $S_k$ and $Y_k$.
It plays the same role for the Kleinmichel formula as the results in \cite{byrd_nocedal_schnabel_1994} do for the BFGS and SR1 updates.
In a limited-memory implementation one simply replaces $S_k$ and $Y_k$ by the matrices containing the $m$ most recent columns, exactly as in the L-BFGS and L-SR1 cases.
The proof is reported in \cref{sec:omitted_proofs}.

\begin{theorem}\label{thm:kleinmichel_compact}
    Let the symmetric matrix $H_0$ be updated $k$ times by means of the Kleinmichel formula using the pairs $\left\{ d^j, y^j \right\}_{j=0}^{k-1}$, and assume that each update is well-defined, namely $\innerprod{y^j-\gamma_j H_j d^j}{d^j} \neq 0$, $j=0, \ldots, k-1$.
    Then the resulting matrix $H_k$ is given by 
    \begin{align}
        \label{eq_KM_C}
        H_k = \overline{\gamma}_k H_0 + Q_k M_k^{-1} Q_k^\top,
    \end{align}
    where $S_k$ and $Y_k$ are defined in \eqref{eq:kleinmichel_Sk_Yk}, $\overline{\gamma}_k \coloneqq \prod_{i=0}^{k-1} \gamma_i$ for all $k \geq 1$, the nonsingular matrix $M_k$ is defined recursively by
    \[
    M_1 = \innerprod{q^0}{d^0}, \quad
    M_{j+1} = \begin{bmatrix}
        \gamma_j^{-1} M_j & Q_j^\top d^j \\
        (Q_j^\top d^j)^\top & \innerprod{q^j}{d^j}
    \end{bmatrix}, \quad j \geq 1 ,
    \]
    and $Q_k \coloneqq [q^0, \ldots, q^{k-1}] \coloneqq Y_k - \overline{\gamma}_{k} H_0 S_k$.
\end{theorem}

\section{Algorithmic details}\label{sec:algo_refinements}

This section collects design choices and features for our implementation \labelRPQN{} of \cref{alg:RPQN}.

\subsection{Termination condition}

\cref{alg:RPQN} is complemented with a criterion for declaring that an iterate $x^k$ is sufficiently close to stationarity.
Upon a successful iteration, we compute the residual
\begin{equation}
    r_k \coloneqq \mu_k \| x^{k+1}-x^k \| = \mu_k \|d^k\|,
\end{equation}
which provides a metric for monitoring the quality of $x^k$; see \cref{lemma:pred} and \cite{stella2017simple,kanzow_mehlitz_2022}.
As soon as the residual $r_k$ falls below a user-specified tolerance, \cref{alg:RPQN} returns the current iterate $x^k$ and terminates.

\subsection{Skipping quasi-Newton updates}

Since the BFGS and Kleinmichel updates remain positive definite when $\innerprod{d^k}{y^k} > 0$,
it is standard practice to skip the update whenever
\begin{equation*}
	\innerprod{d^k}{y^k} < \varepsilon_{\textrm{QN}} \|d^k\|^2
\end{equation*}
for some prescribed $\varepsilon_{\textrm{QN}} > 0$.
In particular, for the Kleinmichel update we set $\gamma_k = \frac{\innerprod{y^k}{d^k}}{2\innerprod{d^k}{H_k d^k}}$ at each iteration $k \in \N_0$, in accordance with \cref{theorem_km_pd}.
Throughout the numerical experiments, we set $\varepsilon_{\textrm{QN}} = 10^{-8}$.
In the case of the SR1 update, poorly conditioned steps are automatically excluded by a strategy described in \cite[Section~3.3.2]{lechner_2022}.

The initialization matrix $H_{k,0}$ for the limited-memory quasi-Newton updates is chosen as
\[
	H_{k,0} = \frac{\innerprod{y^k}{y^k}}{\innerprod{d^k}{y^k}} I,
\]
following the recommendation of Liu \& Nocedal \cite{liu_nocedal_1989}.
In particular, when no curvature pairs are stored (i.e., zero memory), the resulting algorithm reduces to a pure proximal-gradient method.

\subsection{Parameter updates}\label{sec:tr_update}

\cref{alg:RPQN} does not specify how to update the regularization parameter $\mu$ at successful iterations; it permits any value in the user-specified interval $[\mumin,\mumax]$.
Our implementation \labelRPQN{} follows classical trust-region methods, distinguishing between successful and highly successful steps.
Specifically, \lineref{algo_line_K_commands_2} is executed according to
\begin{equation}
	\mu_{k+1} = \begin{cases}
		\sigma_1 \mu_k &\text{if } \ared_k \geq c_2 \pred_k, \\
		\mu_k &\text{otherwise}
	\end{cases}
\end{equation}
for some fixed $\sigma_1 \in (0,1)$ and $c_2 \in (c_1, 1)$.
Included also in R2N \cite{diouane2026proximal}, this strategy allows to reduce the regularization, possibly speeding up convergence, when the model seems to fit the actual problem particularly well.
Otherwise, the regularization parameter is kept constant to avoid unwarranted unsuccessful iterations.

\subsection{Nonmonotone globalization}\label{sec:nonmonotone_glob}

In order to impose less conservative steps, with the aim of reducing the number of unsuccessful iterations, we incorporate a nonmonotone globalization mechanism.
While \cite{birgin2000nonmonotone,kanzow_mehlitz_2022,diouane2026proximal} adopt the \emph{max} kind introduced by Grippo, Lampariello \& Lucidi \cite{grippo1986nonmonotone},
we use instead the \emph{average} scheme investigated in \cite{zhang_hager_2004,themelis_stella_patrinos_2018,demarchi_2023}, which does not interfere with convergence properties of the monotone counterpart (particularly \cref{theorem:mu_pred_3,thm:liminf_stationarity,theorem:global_convergence_1}).

Let $\nmfactor \in (0,1]$ be a given monotonicity factor and, for all $k\in\N$, define recursively the merit sequence $\{\Phi_k\}$ by
\[
\Phi_k \coloneqq \begin{cases}
	F(x^0) &\text{if}~k=0 , \\
	\nmfactor F(x^k) + (1-\nmfactor) \Phi_{k-1} &\text{if}~k>0 .
\end{cases}
\]
Then, by evaluating the actual improvement $\ared_k$ at \lineref{algo_line_ared_pred} according to
\[
\ared_k \coloneqq \Phi_k - F(\hat{x}^k),
\]
\cref{alg:RPQN} enforces sufficient decrease with respect to the merit $\Phi_k$ and not the current objective $F(x^k)$.
The scheme falls back to monotone decrease when $\nmfactor = 1$, otherwise the fact that $\Phi_k \geq F(x^k)$ by \cite[Lemma 4.2]{demarchi_2023} relaxes the condition for accepting a tentative update $\hat{x}^k$, allowing larger steps and possibly faster convergence in practice.

\section{Numerical results}\label{Sec:Numerics}
\newcommand{\best}[1]{\colorbox{lightgray}{\textbf{#1}}}

In this section, we report numerical experiments for our implementation \texttt{RPQN} of \cref{alg:RPQN} on several instances of \eqref{problem}.
We evaluate algorithmic variants with three limited-memory quasi-Newton updates (BFGS, SR1 and Kleinmichel) and (non)monotone globalization mechanism.
Our experiments consider convex logistic regression and nonconvex Student’s $t$-regression models with either convex or nonconvex regularizers.\footnote{The proximal mapping of these regularizers can be found online at \href{https://proximity-operator.net/}{proximity-operator.net}.}
We analyze the impact of different quasi-Newton updates and of accuracy requirements.
	
\labelRPQN{} is compared against three other solvers designed to tackle \eqref{problem}.
\labelSPG{} implements a proximal-gradient method with spectral (or Barzilai--Borwein) stepsize \cite{birgin2000nonmonotone,jia_kanzow_mehlitz_2023,demarchi_2023},
which often performs better than proximal-gradient methods with simple backtracking.
\labelPANOC{} is a proximal-gradient method that incorporates quasi-Newton steps to speed up convergence \cite{stella2017simple,demarchi2022proximal}.
Finally, \labelRtwoN{} is a proximal quasi-Newton method that combines a second-order model with an inexact inner solve and an adaptive quadratic regularization \cite{diouane2026proximal}.
Our implementation of solvers \labelSPG{}, \labelPANOC{} and \labelRtwoN{} closely follows the original descriptions in \cite{demarchi_2023}, \cite{demarchi2022proximal} and \cite{diouane2026proximal,gollier2026regularizedoptimization}, respectively, with minimal modifications to ensure comparability across methods.
In particular, all solvers have a common termination criterion and (non)monotone globalization strategy, as presented in \cref{sec:algo_refinements}.
\labelPANOC{} and \labelRtwoN{} use their default limited-memory BFGS with memory $m=5$.
A summary of all solver configurations is given in \cref{tab:solvers}.

 \begin{table}[tbh]
    \centering
    \begin{tabular}{cc}
         \hline
         \labelSPG & proximal-gradient method with spectral stepsize \cite{demarchi_2023} \\
         \labelRtwoN & proximal quasi-Newton method with limited-memory BFGS \cite{diouane2026proximal} \\
         \labelPANOC & PANOC$^+$ with limited-memory BFGS \cite{demarchi2022proximal} \\
         \hline
         \labelRPQNlkm & RPQN with limited-memory Kleinmichel \\
         \labelRPQNlsrone & RPQN with limited-memory SR1 \\
         \labelRPQNlbfgs & RPQN with limited-memory BFGS \\
         \hline
         \texttt{...}\labelnm & variant with averaged nonmonotone globalization \\
         \hline
    \end{tabular}
    \caption{Solver configurations considered for comparison in our numerical experiments.}
    \label{tab:solvers}
\end{table}

For the sake of comparison, all solvers use the same stopping criteria:
residual $r_k$ smaller than a tolerance ($\texttt{tol}>0$)
or exceeded time limit (300 wall-clock seconds).\footnote{All computations were performed sequentially in MATLAB R2025b on a 64-bit Linux laptop.
The hardware configuration was an Intel Core i7 processor, 16 GB RAM.
Timings are intended for relative comparison only.}
Our tests consider both low ($\texttt{tol}=10^{-3}$) and high ($\texttt{tol}=10^{-5}$) accuracy.

All \texttt{RPQN} solvers run using an identical set of parameters, which we found to work well and are not fine-tuned:
$c_1 = 10^{-4}$, $c_2 = \nicefrac{9}{10}$, $\sigma_1 = \nicefrac{1}{2}$, $\sigma_2 = 4$; see \cref{sec:tr_update}.
For the nonmonotone variants, we set $\nmfactor = \nicefrac{1}{10}$.
For all limited-memory updates, we set a memory of $m=5$.
For the solution of subproblems (discussed in \cref{sec_subproblems}),
we choose a tolerance of $10^{-9}$ as a stopping criterion for the root-finding problem.
 
\medskip

Numerical results are reported as median values in tables, where we highlight the method with the best performance for each problem class and accuracy level.
Solvers are also compared in terms of wall-clock runtime and final objective by means of profiles.
Given a set $P$ of problem instances and a set $S$ of solvers, let $\tau_{s,p}$ denote the performance metric (lower is better) of solver $s\in S$ for problem $p\in P$.
We adopt \emph{absolute profiles} to graphically compare different solvers.
Each profile depicts the empirical cumulative distribution $\varrho_s \colon \R \mapsto [0, 1]$ of $\tau_{s,\cdot}$ over $P$, namely
\[
\forall \kappa \in \R \colon\quad
\varrho_s(\kappa) \coloneqq \frac{|\{ p\in P \,|\, \tau_{s,p} \leq \kappa \}|}{|P|},
\]
where the sample size $|P|$ is the cardinality of set $P$.
As such, absolute profiles are independent of each other and display the fraction of problems $\varrho_s(\kappa)$ solved by $s$ with performance metric at most $\kappa$.

\paragraph{Results}
Across all problem classes below, the \labelRPQN\labelnm{} variants perform comparably to one another and consistently outperform \labelSPG\labelnm{}, \labelRtwoN\labelnm{}, and \labelPANOC\labelnm{} in terms of runtime, with the gap widening at high accuracy.
However, \labelPANOC\labelnm{} and \labelSPG\labelnm{} tend to find the lowest objective values, followed by \labelRPQN\labelnm{} variants and then \labelRtwoN\labelnm{}.

All methods appear to benefit from the nonmonotone globalization mechanism, for both low and high accuracy, in terms of runtime, with marginal differences in objective values.
Since nonmonotone variants performed better than their monotone counterparts in most cases, results are reported only for the \labelnm{} solvers.
	
\subsection{Regularized logistic regression}\label{sec:logreg}
\newcommand{\nfeat}{n_{\textrm{f}}}
\newcommand{\nsamples}{n_{\textrm{s}}}

A common model for solving sparse binary classification problems is regularized logistic regression.
Given feature vectors $a_i\in\R^{\nfeat}$ and labels $b_i\in\{-1,1\}$ for $i=1,\dots,\nsamples$, we aim to
\begin{equation}\label{eq:log_reg_l1}
	\minimize_{y,v}\quad
	\frac{1}{\nsamples}\sum_{i=1}^{\nsamples} \log \left( 1+\exp \bigl( -b_i(a_i^\top y+v) \bigr) \right) + \varphi(y)
\end{equation}
with respect to $y\in\R^{\nfeat}$ and $v\in\R$.
In common test instances there are more samples than features ($\nsamples \gg \nfeat$).
A sparse representation is encouraged by including a regularizer $\varphi$.
First, we let $\varphi(y) \coloneqq \lambda\|y\|_1$ for some given regularization parameter $\lambda>0$.
Then, we consider the (nonconvex) capped-$\ell_1$ regularizer, a piecewise linear approximation of $\ell_0$ given by
$\varphi(y) \coloneqq \lambda \sum_{i=1}^{\nfeat} g_1(y_i)$ where $g_1(t) \coloneqq \min\{|t|,1\}$.
Results for the nonmonotone variants are reported in \cref{fig:logreg,tab:logreg}.

\paragraph{Setup}
We generate sparse instances with $\nfeat=10^4$ and $\nsamples=10^5$, and vary the feature sparsity parameter $s\in\{10,100\}$.
For each $i$, the vector $a_i$ has approximately $s$ nonzero components; the locations are chosen at random and the nonzero entries are drawn independently from a standard normal distribution $\mathcal N(0,1)$.
We sample a ground-truth vector $y^{\mathrm{true}}\in\R^{\nfeat}$ with $10s$ nonzero entries, drawn again from $\mathcal N(0,1)$, and a bias $v^{\mathrm{true}}\sim\mathcal N(0,1)$.
Labels are then assigned via
$b_i = \sign\left(a_i^\top y^{\mathrm{true}} + v^{\mathrm{true}} + \xi_i\right)$,
where $\xi_i\sim\mathcal N(0,\nicefrac{1}{10})$ independently for $i=1,\dots,\nsamples$.

We parametrize $\lambda$ as $\lambda \coloneqq c_\lambda \lambda_{\max}$ with $c_\lambda\in\{ 10^{-1}, 10^{-2}, 10^{-3} \}$.
Following \cite{koh_2007}, we denote by $\nsamples^+$ and $\nsamples^-$ the number of indices with $b_i=1$ and $b_i=-1$, respectively, and choose
\[
\lambda_{\max}
= \frac{1}{\nsamples}\left\|
\frac{\nsamples^-}{\nsamples} \sum_{i:\,b_i=1} a_i + \frac{\nsamples^+}{\nsamples} \sum_{i:\,b_i=-1} a_i
\right\|_\infty ,
\]
which is the smallest value for which the trivial solution $y^\ast=0$ (with some $v^\ast$) is optimal.

For each combination of $(s,c_\lambda)$, we generate 10 independent data sets and run each solver variant, starting from an all-zero vector $x^0\in\R^{\nfeat+1}$.

\begin{figure}[tbh]
	\centering%
	\includegraphics{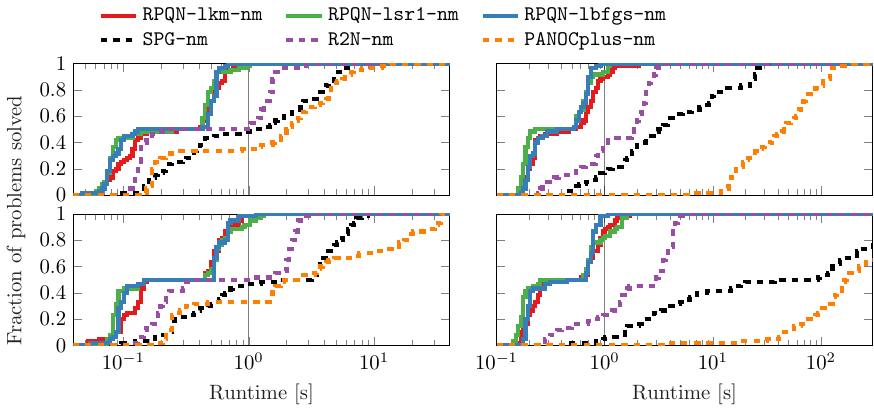}\\
	\includegraphics{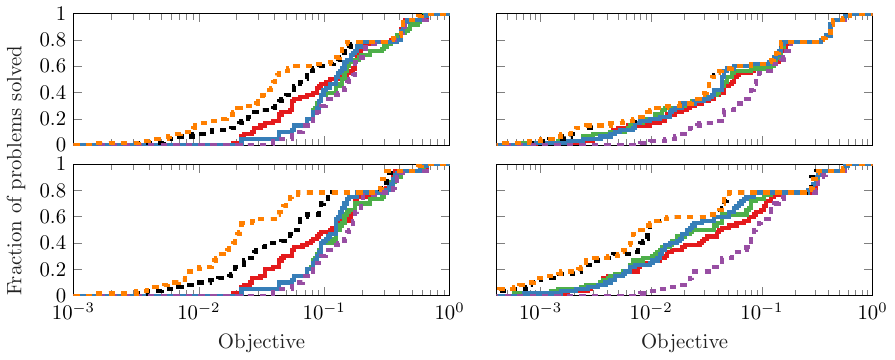}%
	\caption{Comparison on logistic regression \eqref{eq:log_reg_l1}, in terms of runtime and objective value. Accuracy: low (left) and high (right). Regularizer: $\ell_1$ (top) and capped $\ell_1$ (bottom).}%
	\label{fig:logreg}%
\end{figure}

\begin{table}[tbh]
    \centering
    \begin{tabular}{cc|ccccc|ccccc}
    	\hline
    	& & \multicolumn{5}{c|}{low accuracy ($\texttt{tol}=10^{-3}$)} & \multicolumn{5}{c}{high accuracy ($\texttt{tol}=10^{-5}$)} \\
    	& & obj. & runtime & \#$\nabla f$ & \#$\varphi$ & fail. & obj. & runtime & \#$\nabla f$ & \#$\varphi$ & fail. \\
    	\hline
    	\multirow{6}{*}{\rotatebox{90}{$\ell_1$}}
    	&\labelSPG\labelnm & 0.07 & 1.2 & 40 & 60 & 0 & \best{0.035} & 3.2 & 129 & 217 & 0 \\ 
    	&\labelRtwoN\labelnm & 0.16 & 0.61 & 15 & 907 & 0 & 0.088 & 1.9 & 20 & 4372 & 0 \\ 
    	&\labelPANOC\labelnm & \best{0.04} & 2 & 74 & 37 & 0 & \best{0.035} & 46 & 2990 & 1495 & 0 \\ 
    	&\labelRPQNlkm\labelnm & 0.12 & 0.34 & 15 & 34 & 0 & 0.055 & 0.49 & 25 & 86 & 0 \\ 
    	&\labelRPQNlsrone\labelnm & 0.14 & 0.33 & 13 & 38 & 0 & 0.048 & \best{0.36} & 21 & 69 & 0 \\ 
    	&\labelRPQNlbfgs\labelnm & 0.12 & \best{0.27} & 14 & 35 & 0 & 0.043 & 0.43 & 22 & 65 & 0 \\ 
    	\hline
    	\multirow{6}{*}{\rotatebox{90}{capped $\ell_1$}}
    	&\labelSPG\labelnm & 0.06 & 2.4 & 55 & 92 & 0 & 0.009 & 64 & 1463 & 2972 & 23\% \\ 
    	&\labelRtwoN\labelnm & 0.16 & 0.69 & 15 & 959 & 0 & 0.086 & 2.9 & 21 & 5469 & 0 \\ 
    	&\labelPANOC\labelnm & \best{0.02} & 2.2 & 127 & 63 & 0 & \best{0.008} & 163 & 9629 & 4814 & 35\% \\ 
    	&\labelRPQNlkm\labelnm & 0.11 & \best{0.29} & 15 & 35 & 0 & 0.042 & 0.48 & 25 & 101 & 0 \\ 
    	&\labelRPQNlsrone\labelnm & 0.13 & \best{0.29} & 13 & 38 & 0 & 0.029 & \best{0.44} & 23 & 84 & 0 \\ 
    	&\labelRPQNlbfgs\labelnm & 0.12 & 0.32 & 15 & 38 & 0 & 0.021 & 0.54 & 23 & 83 & 0 \\ 
    	\hline
    \end{tabular}
    \caption{Comparison on logistic regression \eqref{eq:log_reg_l1}, in terms of median objective value, runtime [s], number of gradient evaluations (\#$\nabla f$) and proximal calls (\#$\varphi$), and failure rate.}
    \label{tab:logreg}
\end{table}

\subsection{Regularized Student's \texorpdfstring{$t$}{t}-regression}\label{sec:studreg}

We consider the regularized Student's $t$-regression problem, which seeks to
\begin{equation}\label{eq:studreg_l1}
	\minimize_{x}\quad
	\frac{1}{m}\sum_{i=1}^m \log \left( 1+\frac{(Ax-b)_i^2}{\nu} \right) + \varphi(x),
\end{equation}
with respect to $x\in\R^n$, where $A\in\R^{m\times n}$, $b\in\R^m$, and $\nu>0$.
The loss term is smooth, but generally nonconvex, while the regularizer $\varphi$ is added to promote sparsity.
As before for the logistic regression, we let the regularizer $\varphi$ be the $\ell_1$ or capped-$\ell_1$ penalty scaled by some weight $\lambda>0$.
\cref{fig:studreg,tab:studreg} summarize the results obtained.

\paragraph{Setup}
The test instances are generated as in \cite{liu_pan_wu_yang_2024}.
We set $m \coloneqq 1024$, $n \coloneqq 8m$, and $\nu \coloneqq \nicefrac{1}{4}$.
Matrix $A$ is defined via subsampled discrete cosine measurements: let $D\in\R^{n\times n}$ denote the discrete cosine transform (DCT) matrix, pick an index set $J\subset\{1,\dots,n\}$ uniformly at random with $|J|=m$, and define $Ax \coloneqq (Dx)_J$.
We generate a sparse ground-truth signal $x^{\mathrm{true}}\in\R^n$ with $s \coloneqq \lfloor \frac{n}{40} \rfloor$ nonzero entries at uniformly random locations.
The nonzero amplitudes are chosen according to
$x_i^{\mathrm{true}} \coloneqq \eta_{1,i} 10^{\eta_{2,i} d}$,
where $\eta_{1,i}\in\{-1,1\}$ is a random sign and $\eta_{2,i}$ is uniformly distributed on $[0,1]$.
The parameter $d\in\{2,3,4\}$ controls the dynamic range.
The observation vector is then given by $b \coloneqq A x^{\mathrm{true}} + \nicefrac{\xi}{10}$,
where the components of $\xi$ are drawn independently from a Student's $t$-distribution with $4$ degrees of freedom.
We choose the regularization parameter as $\lambda \coloneqq c_\lambda \|\nabla f(0)\|_\infty$ with $c_\lambda\in\{ 10^{-1},10^{-2}\}$.
For each combination of $d$ and $c_\lambda$ we generate 10 independent instances and start all methods from $x^0 \coloneqq A^\top b$.

\begin{figure}[tbh]
	\centering%
	\includegraphics{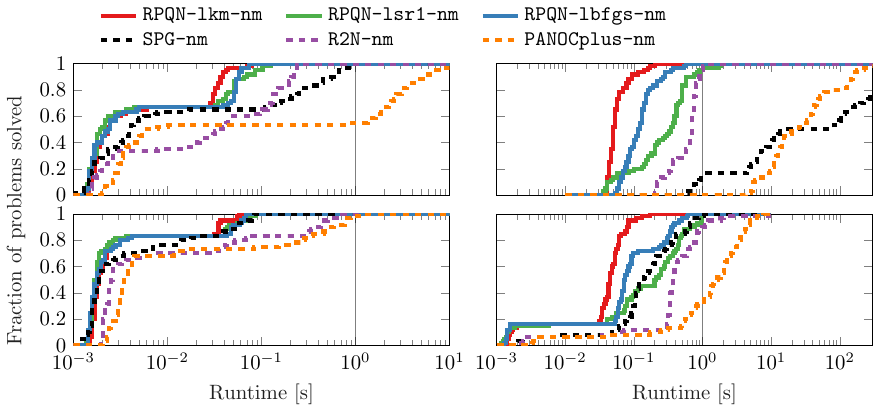}\\
	\includegraphics{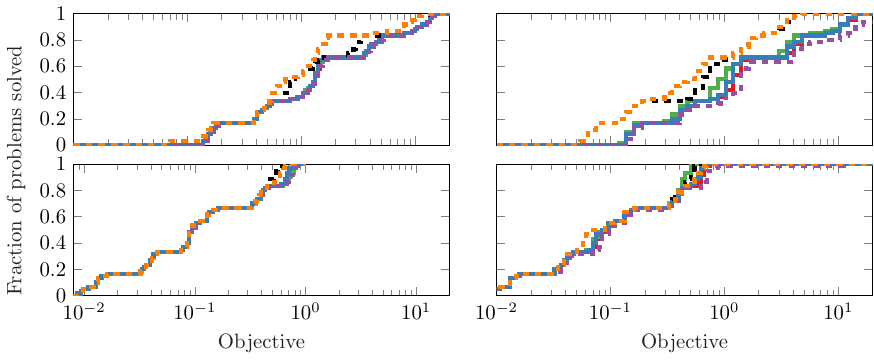}%
	\caption{Comparison on Student's $t$-regression \eqref{eq:studreg_l1}, in terms of runtime and objective value. Accuracy: low (left) and high (right). Regularizer: $\ell_1$ (top) and capped $\ell_1$ (bottom).}%
	\label{fig:studreg}%
\end{figure}

\begin{table}[tbh]
	\centering
	\begin{tabular}{cc|ccccc|ccccc}
		\hline
		& & \multicolumn{5}{c|}{low accuracy ($\texttt{tol}=10^{-3}$)} & \multicolumn{5}{c}{high accuracy ($\texttt{tol}=10^{-5}$)} \\
		& & obj. & runtime & \#$\nabla f$ & \#$\varphi$ & fail. & obj. & runtime & \#$\nabla f$ & \#$\varphi$ & fail. \\
		\hline
		\multirow{6}{*}{\rotatebox{90}{$\ell_1$}}
		&\labelSPG\labelnm & 0.84 & 0.004 & 4 & 4 & 0 & 0.63 & 48 & 44773 & 65199 & 20\% \\ 
		&\labelRtwoN\labelnm & 1.3 & 0.035 & 12 & 37 & 0 & 1.3 & 0.67 & 21 & 5181 & 0 \\ 
		&\labelPANOC\labelnm & \best{0.73} & 0.006 & 6 & 3 & 0 & \best{0.47} & 24.7 & 32327 & 16164 & 1.7\% \\ 
		&\labelRPQNlkm\labelnm & 1.3 & \best{0.002} & 2 & 1 & 0 & 1.2 & \best{0.05} & 24 & 148 & 0 \\ 
		&\labelRPQNlsrone\labelnm & 1.3 & \best{0.002} & 2 & 1 & 0 & 0.87 & 0.33 & 97 & 1539 & 0 \\ 
		&\labelRPQNlbfgs\labelnm & 1.3 & \best{0.002} & 2 & 1 & 0 & 1 & 0.12 & 30 & 221 & 0 \\ 
		\hline
		\multirow{6}{*}{\rotatebox{90}{capped $\ell_1$}}
		&\labelSPG\labelnm & 0.093 & \best{0.002} & 2 & 2 & 0 & \best{0.076} & 0.14 & 96 & 148 & 0 \\ 
		&\labelRtwoN\labelnm & 0.093 & 0.003 & 2 & 3 & 0 & 0.095 & 0.37 & 17 & 1424 & 1.7\% \\ 
		&\labelPANOC\labelnm & \best{0.09} & 0.003 & 4 & 2 & 0 & \best{0.076} & 1.5 & 1681 & 840 & 0 \\ 
		&\labelRPQNlkm\labelnm & 0.093 & \best{0.002} & 2 & 1 & 0 & 0.085 & \best{0.044} & 20 & 96 & 0 \\ 
		&\labelRPQNlsrone\labelnm & 0.093 & \best{0.002} & 2 & 1 & 0 & 0.083 & 0.2 & 50 & 744 & 0 \\ 
		&\labelRPQNlbfgs\labelnm & 0.093 & \best{0.002} & 2 & 1 & 0 & 0.083 & 0.07 & 21 & 115 & 0 \\ 
		\hline
	\end{tabular}
	\caption{Comparison on Student's $t$-regression \eqref{eq:studreg_l1}, in terms of median objective value, runtime [s], number of gradient evaluations (\#$\nabla f$) and proximal calls (\#$\varphi$), and failure rate.}%
	\label{tab:studreg}
\end{table}

\section{Conclusion}\label{Sec:Final}
We proposed a proximal quasi-Newton method for nonsmooth nonconvex optimization and established global convergence under weak assumptions, together with convergence of the iterates and rates under the Kurdyka--\L{}ojasiewicz property.
We also developed a compact representation of the Kleinmichel rank-one update and demonstrated the practical performance of the resulting algorithm on regularized regression problems.

Several questions remain open.
First, our analysis does not establish whether all iterations eventually become successful.
Determining whether unsuccessful iterations can occur infinitely often, particularly near a limit point, would clarify the algorithm's asymptotic behavior.
Second, a more precise characterization of the relationship between the subproblem solver for nonconvex regularizers and the local proximity operator would strengthen the theoretical foundations of the approach and clarify the interpretation of the computed solutions.
Third, the Kleinmichel update warrants further investigation: it preserves positive definiteness under the same condition as BFGS, it admits an efficient compact representation, and our experiments suggest that it can also yield improved performance in practice.

\appendix

\section{Omitted proofs}\label{sec:omitted_proofs}

\subsection*{Proof of \cref{lemma_dist}}
\begin{proof}
Since $k\in\mathcal S$, $x^{k+1} = \hat x^k$ solves \eqref{eq_subproblem} in the sense of \eqref{eq:localprox}, whose stationarity condition gives
\[
0 \in \nabla f(x^k) + G_k (x^{k+1} - x^k) + \partial \varphi (x^{k+1}).
\]
Hence
\begin{equation}\label{eq_Bk_in_subdiff_F}
	-G_k (x^{k+1} - x^k) + \nabla f(x^{k+1}) - \nabla f(x^k) \in 
	\nabla f(x^{k+1}) + \partial \varphi (x^{k+1}) = \partial F(x^{k+1}),
\end{equation}
where we used the sum rule for the limiting subdifferential.
Then, by $\Vert d^k \Vert \leq \varepsilon_0 - \varepsilon_1$, it holds that $x^{k+1} \in \ball_{\varepsilon_0}(x^*)$.
\cref{ass:Bk_bounded,lemma:mu_bounded} guarantee that $\|G_k\| \leq M_B+\mubar$.
Then, using \eqref{eq_Bk_in_subdiff_F} together with \cref{ass:loc_lipschitz}, we obtain
\begin{align*}
	\dist\left(0,\partial F(x^{k+1})\right)
	\leq{}&
	\Vert -G_k(x^{k+1}-x^k) + \nabla f(x^{k+1}) - \nabla f(x^k) \Vert
	\\
	\leq{}&
	(M_B+\mubar) \Vert x^{k+1}-x^k \Vert + L^* \Vert x^{k+1} - x^k \Vert
	={}
	(M_B+\mubar + L^*) \Vert x^{k+1}-x^k \Vert, 
\end{align*}
which completes the proof.
\end{proof}

\subsection*{Proof of \cref{theorem_final}}
\begin{proof}
	The sequence $\{F(x^k)\}$ is monotonically decreasing (\cref{lemma:basic_properties}\ref{lemma:basic_F}) and converges to $F(x^*)$ by \cref{assumption_loc_prox_cvx}\ref{ass_existence_cvx} (recall \cref{Prop:Fxk-converges}); in particular, $F(x^k) \geq F(x^*)$ for all $k \in \N_0$.
	Now suppose $F(x^k) = F(x^*)$ for some $k \in \N_0$. By monotonicity, also $F(x^{k+1}) = F(x^*)$.
	First assume that iteration $k$ is successful.
	Then
	\begin{equation}\label{eq_1_final_theorem}
		0 = F(x^k)-F(x^{k+1})
		=
		\ared_k
		\geq
		c_1 \pred_k \geq \frac{c_1 \mumin}{2} 
		\Vert x^{k+1} - x^k \Vert^2 \geq 0,
	\end{equation}
	where the second inequality uses \cref{lemma:basic_properties}\ref{lemma:basic_pred}. 
	Hence $x^{k+1}=x^k$, and this holds trivially if iteration $k$ is unsuccessful, i.e., the sequence $\{x^k\}$ is eventually constant.
	Since, by \cref{assumption_loc_prox_cvx}\ref{ass_existence_cvx}, a subsequence of $\{x^k\}_{\mathcal S}$ converges to $x^*$, it follows that $x^k=x^*$ for all sufficiently large $k \in \N_0$.
	
	For the remainder of the proof, assume $F(x^k) > F(x^*)$ for all $k \in \N_0$.
	According to \cref{lemma_alpha}, there exists $\mathcal S' \subset \mathcal S$ with $\{\alpha_k\}_{\mathcal S'} \to 0$.
	Together with \cref{theorem:global_convergence_cvx} and $F(x^k)\to_{\mathcal{S}}F(x^*)$, we can choose $k_0 \in \mathcal S'$ sufficiently large such that
	\begin{equation}\label{eq_F}
		\alpha_{k_0} \in (0,\varepsilon_1), \quad \ball_{\alpha_{k_0}}(x^*) \subset U, \quad F(x^{k_0}) < F(x^*)+\eta,
	\end{equation}
	and $\Vert d^{k} \Vert \leq \varepsilon_0-\varepsilon_1$ for all successful iterations $k \geq k_0$, where $ \eta $ denotes the constant from the desingularization function $\chi \colon [0,\eta] \to [0,\infty)$ associated with the KL property at $x^*$.
	Since $\chi (0) = 0$ and $\chi'(t) > 0$ for all $t \in (0,\eta)$,
	$\chi \left(F(x^k) - F(x^*) \right) \geq 0$ for all $k \geq k_0$.
	Denote by $s_0,s_1,s_2,\ldots$ the successful iterations, starting with $s_0 \coloneqq k_0$.
	As iterates are not updated during unsuccessful iterations,
	\begin{equation}\label{eq_xxs}
		x^{s_{j+1}} = x^{s_j+1} \quad \text{for all } j \in \N_0.
	\end{equation}
	We claim that for all $j \in \N_0$:
	\begin{enumerate}[label=(\alph*)]
		\item $x^{s_j} \in \ball_{\alpha_{k_0}} (x^*)$, \label{eq_induction_a}
		\item $\Vert x^{k_0} - x^* \Vert + \sum\limits_{i=0}^{j} \Vert
		x^{s_{i+1}}-x^{s_i} \Vert \leq \alpha_{k_0}$, which is equivalent to
		\begin{equation}\label{eq_induction_b_1}
			\sum\limits_{i=0}^{j} \Vert x^{s_{i+1}}-x^{s_i} \Vert 
			\leq
			\sqrt{\frac{8}{c_1\mumin}\big( F(x^{k_0})-F(x^*) \big)} + \frac{1}{\beta}\chi \left( F(x^{k_0})-F(x^*) \right).
		\end{equation} \label{eq_induction_b}
	\end{enumerate}
	We prove \ref{eq_induction_a}--\ref{eq_induction_b} by induction.
	For $j=0$, \ref{eq_induction_a} holds by the definition of $\alpha_{k_0}$.
	Moreover, \eqref{eq_1_final_theorem} together with the monotonicity of $\{F(x^k)\}$ and \eqref{eq_xxs} yields
	\begin{equation}\label{eq_k0}
		\Vert x^{s_1} - x^{s_0} \Vert
		=
		\Vert x^{k_0+1} - x^{k_0} \Vert 
		\leq
		\sqrt{\frac{2}{c_1\mumin}\big( F(x^{k_0})-F(x^{k_0+1})\big)} 
		\leq
		\sqrt{\frac{2}{c_1\mumin}\big( F(x^{k_0})-F(x^*)\big)}.
	\end{equation}
	Hence \ref{eq_induction_b} also holds for $j=0$.
	Assume \ref{eq_induction_a}--\ref{eq_induction_b} hold for $i=0,\dots,j$ with some $j \ge 0$.
	By the triangle inequality and the induction hypothesis,
	\begin{equation}\label{eq_a}
		\Vert x^{s_{j+1}}-x^* \Vert \leq \Vert x^{k_0}-x^* \Vert +
		\sum_{i=0}^{j} \Vert x^{s_{i+1}}-x^{s_i} \Vert \leq \alpha_{k_0},
	\end{equation}
	so \ref{eq_induction_a} holds with $j+1$ in place of $j$.
	For \ref{eq_induction_b}, first note that \eqref{eq_F} implies
	$F(x^*) < F(x^k) < F(x^*)+\eta$
	for all $k \geq k_0$.
	Since $F$ has the KL property at $x^*$ and $x^{s_i} \in \ball_{\alpha_{k_0}} (x^*) \subset U$ for all $i \in \{0,\dots,j+1\}$ by the induction hypothesis and \eqref{eq_a}, we have
	\begin{equation}\label{eq_KL}
		\chi' \big( F(x^{s_i})-F(x^*) \big) 
		\dist \big( 0,\partial F(x^{s_i}) \big) \geq 1 
	\end{equation}
	for all $i = 0, \hdots, j+1$.
	\cref{lemma_dist} and \eqref{eq_xxs} give
	\begin{multline*}
		\dist \big( 0,\partial F(x^{s_{i+1}}) \big) 
		=
		\dist \big( 0,\partial F(x^{s_i+1}) \big) 
		\leq
		(M_B+\mubar+L^*) \Vert x^{s_i+1}-x^{s_i} \Vert
		\\
		=
		(M_B+\mubar+L^*) \Vert x^{s_{i+1}}-x^{s_i} \Vert
	\end{multline*}
	for all $i = 0,\dots,j+1$.
	In view of \eqref{eq_KL}, this yields
	\begin{equation}\label{eq_chi_ineq}
		\chi' \big( F(x^{s_{i}})-F(x^*) \big)
		\geq 
		\frac{1}{\dist \big( 0,\partial F(x^{s_{i}}) \big)}
		\geq
		\frac{1}{(M_B+\mubar+L^*) \Vert x^{s_{i}}-x^{s_{i-1}} \Vert}
	\end{equation}
	for all $i=1,\dots,j+2$.
	For convenience, set
	\begin{equation*}
		\Delta_{i,l}
		\coloneqq
		\chi \big( F(x^{s_i})-F(x^*) \big) - \chi \big( F(x^{s_l})-F(x^*) \big)
	\end{equation*}
	for $i,l = 0,\dots,j+2$.
	Then, by concavity of $\chi$, we obtain
	\begin{multline*}
		\Delta_{i,i+1}
		\geq
		\chi' \big( F(x^{s_i})-F(x^*) \big) \big( F(x^{s_i})-F(x^{s_{i+1}}) \big) 
		\overset{\eqref{eq_chi_ineq}}{\geq}
		\frac{F(x^{s_i})-F(x^{s_i+1})}{(M_B+\mubar+L^*)\Vert x^{s_i}-x^{s_{i-1}} \Vert }
		\\
		\overset{\eqref{eq_1_final_theorem}}{\geq}
		\frac{c_1 \mumin \Vert x^{s_i+1}-x^{s_i} \Vert^2}{2(M_B+\mubar+L^*) \Vert x^{s_i}-x^{s_{i-1}} \Vert }
		\overset{\eqref{eq_xxs}}{=}
		\beta \frac{\Vert x^{s_{i+1}}-x^{s_i} \Vert^2}{\Vert x^{s_i}-x^{s_{i-1}} \Vert}
	\end{multline*} 
	for all $i \in \{1,\dots,j+1\}$, where $\beta$ is the constant from \cref{lemma_alpha}.
	Since $a+b \geq 2 \sqrt{ab}$ for all $a,b \geq 0$, it follows that
	\begin{equation*}
		\frac{1}{\beta} \Delta_{i,i+1} + \Vert x^{s_i}-x^{s_{i-1}} \Vert 
		\geq
		2 \sqrt{\frac{1}{\beta}\Delta_{i,i+1} \Vert x^{s_i}-x^{s_{i-1}} \Vert } \geq 2 \Vert x^{s_{i+1}}-x^{s_i} \Vert
	\end{equation*}
	for all $i \in \{1,\dots,j+1\}$.
	Summation gives
	\begin{multline*}
		2 \sum_{i=1}^{j+1} \Vert x^{s_{i+1}}-x^{s_i} \Vert 
		\leq
		\sum_{i=1}^{j+1} \Vert x^{s_i} - x^{s_{i-1}} \Vert  + \frac{1}{\beta} \sum_{i=1}^{j+1} \Delta_{i,i+1}
		=
		\sum_{i=0}^j \Vert x^{s_{i+1}}-x^{s_i} \Vert + \frac{1}{\beta} \Delta_{1,j+2}
		\\
		\leq
		\sum_{i=1}^{j+1} \Vert x^{s_{i+1}}-x^{s_i} \Vert + \Vert x^{s_1} - x^{s_0} \Vert + \frac{1}{\beta} \Delta_{1,j+2}.
	\end{multline*}
	Subtracting the first term on the right, using \eqref{eq_k0}, and the nonnegativity and monotonicity of $\chi$, we obtain
	\begin{equation*}
		\sum_{i=1}^{j+1} \Vert x^{s_{i+1}}-x^{s_i} \Vert
		\leq
		\sqrt{\frac{2}{c_1\mumin} \big( F(x^{k_0})-F(x^*) \big)} + \frac{1}{\beta} \chi \big( F(x^{k_0})-F(x^*) \big).
	\end{equation*}
	Adding $\Vert x^{s_1}-x^{s_0} \Vert$ to both sides and using \eqref{eq_k0} again yields
	\begin{equation*}		
		\sum_{i=0}^{j+1} \Vert x^{s_{i+1}}-x^{s_i} \Vert
		\leq
		\sqrt{\frac{8}{c_1\mumin} \big( F(x^{k_0})-F(x^*) \big)} + \frac{1}{\beta} \chi \big( F(x^{k_0})-F(x^*) \big).
	\end{equation*}
	Thus \ref{eq_induction_b} holds with $j+1$ in place of $j$, completing the induction.
	In particular, \ref{eq_induction_a} implies $x^k \in \ball_{\alpha_{k_0}}(x^*)$ for all successful $k \ge k_0$.
	Since iterates are not updated during unsuccessful iterations, this in fact holds for all $k \ge k_0$.
	Letting $\{k_0\}_{\mathcal S'} \to \infty$, we get convergence of the entire sequence $\{x^k\}$ to $x^*$.
\end{proof}

\subsection*{Proof of \cref{thm:kl-successful-recursion}}
\begin{proof}
	Suppose there exists an index set $\mathcal J \subset \N_0$ with $\{s_{j+1}-s_j\}_{\mathcal J} \to +\infty$.
	Then $\mu_{s_{j+1}} \geq \sigma_2^{s_{j+1}-s_j} \mumin \to_{\mathcal J} +\infty$, a contradiction to \cref{theorem_final} and \cref{lemma:mu_bounded}.
	Hence, there exists $\overline s \in \N$ with $s_{j+1}-s_j \leq \overline s$ for all $j \in \N_0$.
	This means that the number of consecutive unsuccessful iterations after $k_0$ is bounded.
	
	We now show convergence rates for the subsequence of successful iterations.
	Together with \cref{theorem:global_convergence_cvx} and $F(x^k)\to_{\mathcal{S}}F(x^*)$ from \cref{assumption_loc_prox_cvx}\ref{ass_existence_cvx}, \cref{theorem_final} guarantees that there exists $j_0\in\N_0$ such that all successful iterates $k\ge s_{j_0}$ satisfy
	\[
	x^k \in \ball_{\varepsilon_1}(x^*) \cap U \cap
	\{x\in\R^n \mid F(x^*)<F(x)<F(x^*)+\eta\}
	\quad \text{and} \quad
	\Vert d^k \Vert \leq \varepsilon_0 - \varepsilon_1.
	\]
	Throughout, we use relation \eqref{eq_xxs}, which holds since unsuccessful iterations do not change the iterates.
	For every $j\in\N_0$, since $s_j$ is successful, \cref{lemma:basic_properties}\ref{lemma:basic_pred} yields
	\begin{multline}
		\label{eq_rate_1_re}
		F(x^{s_{j}})-F(x^{s_{j+1}})
		=
		F(x^{s_{j}})-F(x^{s_j+1})
		=
		\ared_{s_j} \geq c_1 \pred_{s_j} 
		\geq
		\frac{c_1 \mumin}{2} \,\Vert d^{s_j} \Vert^2
		\\
		=
		\frac{c_1 \mumin}{2}\,\|x^{s_j+1}-x^{s_j}\|^2
		=
		\frac{c_1 \mumin}{2}\,\|x^{s_{j+1}}-x^{s_j}\|^2.
	\end{multline}
	\cref{lemma_dist} gives, for all $j\ge j_0$,
	\begin{multline} 
		\label{eq_rate_2_re}
		\dist\bigl(0,\partial F(x^{s_{j+1}})\bigr)
		=
		\dist\bigl(0,\partial F(x^{s_j+1})\bigr)
		\leq
		(M_B + \mubar + L^*)\, \Vert x^{s_j+1} - x^{s_j} \Vert
		\\
		=
		(M_B + \mubar + L^*)\, \Vert x^{s_{j+1}} - x^{s_j} \Vert.
	\end{multline}
	Since $F$ has the KL property at $x^*$ with
	$\chi(t)=c\,t^{1-\theta}$, we have for all $j\ge j_0$,
	\begin{equation*}
		1
		\leq
		\chi^\prime\left( F(x^{s_{j+1}})-F(x^*) \right)
		\dist\left( 0, \partial F(x^{s_{j+1}}) \right)
		=
		c(1-\theta)\bigl(F(x^{s_{j+1}})-F(x^*)\bigr)^{-\theta}\,
		\dist\bigl(0,\partial F(x^{s_{j+1}})\bigr).
	\end{equation*}
	Equivalently,
	\[
	\dist(0,\partial F(x^{s_{j+1}}))
	\;\ge\; \frac{1}{c(1-\theta)}
	\bigl(F(x^{s_{j+1}})-F(x^*)\bigr)^{\theta}.
	\]
	Combining this with \eqref{eq_rate_2_re} yields
	\begin{equation}\label{eq_rate_4_re}
		\|x^{s_{j+1}}-x^{s_j}\|
		\geq
		\frac{(F(x^{s_{j+1}})-F(x^*))^\theta}{c(1-\theta)(M_B+\mubar+L^*)}.
	\end{equation}
	By defining
	\[
	\tau \coloneqq \left(
	\frac{2c^2(1-\theta)^2(M_B+\mubar+L^*)^2}{c_1 \mumin}
	\right)^{\frac{1}{2\theta}}
	\]
	and using \eqref{eq_rate_4_re} and \eqref{eq_rate_1_re}, we obtain
	\begin{align}
		F(x^{s_{j+1}}) - F(x^*)
		\leq{}&
		\big( c (1-\theta) (M_B+\mubar + L^*) \big)^{\frac{1}{\theta}} \Vert x^{s_{j+1}} - x^{s_j} \Vert^{\frac{1}{\theta}}
		\nonumber\\
		\leq{}&
		\big( c (1-\theta) (M_B+\mubar + L^*) \big)^{\frac{1}{\theta}} \left( \frac{2}{c_1 \mumin}\right)^{\frac{1}{2\theta}} \big( F(x^{s_{j}})-F(x^{s_{j+1}}) \big)^{\frac{1}{2\theta}} 
		\nonumber\\
		={}&
		\tau \big( F(x^{s_{j}})-F(x^{s_{j+1}}) \big)^{\frac{1}{2\theta}} \leq \tau \big( F(x^{s_{j}})-F(x^{*}) \big)^{\frac{1}{2\theta}}. \label{eq_rate_5_re}
	\end{align} 
	This implies Q-superlinear convergence of $\{F(x^{s_j})\}$ when
	$\theta\in(0,\tfrac12)$.
	Then, by denoting
	$a_j \coloneqq F(x^{s_j}) - F(x^*)$,
	\eqref{eq_rate_5_re} gives, for $j\ge j_0$,
	\[
	a_{j+1}^{2\theta}
	\leq
	\tau^{2\theta}(a_j - a_{j+1}),
	\]
	which is precisely the setting of \cref{lemma_technical_rates}.
	Hence, for $\theta = \frac{1}{2}$, $\{F(x^{s_j})\}$ converges Q-linearly to $\{F(x^*)\}$ and for $\theta \in \left( \frac{1}{2}, 1 \right)$ we get the inequality
	\[
	F(x^{s_j}) - F(x^*) \leq C_1 j^{-\frac{1}{2\theta-1}}
	\]
	for some $C_1 > 0$ and $j$ sufficiently large. 
	
	We now verify the statements for the sequence $\{x^{s_j}\}$.
	Because $\{x^k\}$ converges to $x^*$, the statements from the induction in the proof of \cref{theorem_final} remain valid when $k_0$ is replaced by any sufficiently large successful iteration $s_k$.
	In particular, \eqref{eq_induction_b_1} implies
	\[
	\sum_{j=k}^l \|x^{s_{j+1}}-x^{s_j}\|
	\leq
	\sqrt{\frac{8}{c_1 \mumin}\bigl(F(x^{s_k})-F(x^*)\bigr)} + \frac{1}{\beta}\chi\bigl(F(x^{s_k})-F(x^*)\bigr)
	\]
	for all $l\ge k$.
	Hence,
	\begin{align*}
		\Vert x^{s_l} - x^{s_k} \Vert
		\leq{}&
		\sum_{j=k}^{l-1} \Vert x^{s_{j+1}} - x^{s_j} \Vert \leq \sqrt{\frac{8}{c_1\mumin}\big( F(x^{s_k})-F(x^*) \big)} + \frac{1}{\beta}\chi \big( F(x^{s_k})-F(x^*) \big)
		\\
		={}&
		\sqrt{\frac{8}{c_1\mumin}\big( F(x^{s_k})-F(x^*) \big)} + \frac{c}{\beta} \big( F(x^{s_k})-F(x^*) \big)^{1-\theta}
		\\
		\leq{}&
		\left( \sqrt{\frac{8}{c_1 \mumin}} + \frac{c}{\beta} \right) \big( F(x^{s_k})-F(x^*) \big)^{\min\{1/2, 1-\theta\}}.
	\end{align*}
	Letting $l \to \infty$, the convergence rates for $\{x^{s_k}\}$ follow by substituting the
	corresponding rates already established for $\{F(x^{s_k})\}$ in the cases
	$\theta\in(0,\tfrac12)$, $\theta=\tfrac12$, and $\theta\in(\tfrac12,1)$.
\end{proof}

\subsection*{Proof of \eqref{eq:ident1} and \eqref{eq:ident2}}
\begin{proof}
	Write $z=U^\top(x-y)$, so $\Lambda(y,\alpha)=z+\alpha$.
	Then,
	\begin{multline*}
		\|w(\alpha)-y\|_{H_0}^2
		=
		\|x-H_0^{-1}U\alpha-y\|_{H_0}^2
		={}
		\|x-y\|_{H_0}^2 - 2\innerprod{U\alpha}{x-y}+\alpha^\top U^\top H_0^{-1}U\alpha \\
		={}
		\|x-y\|_{H_0}^2 - 2\innerprod{\alpha}{U^\top(x-y)}+\alpha^\top U^\top H_0^{-1}U\alpha
		={}
		\|x-y\|_{H_0}^2 - 2\innerprod{\alpha}{z}+\alpha^\top M \alpha.
	\end{multline*}
	Hence
	\begin{multline*}
		g(y)
		=
		\varphi(y) + \tfrac{1}{2} \|x-y\|_{H_0}^2 + \tfrac{1}{2}\|x-y\|_{UU^\top}^2
		={}
		\varphi(y)+\tfrac12\|x-y\|_{H_0}^2+\tfrac12\|z\|^2
		\\
		={}
		\Phi(y;w(\alpha)) + \tfrac12\|z\|^2 + \innerprod{\alpha}{z} - \tfrac{1}{2}\alpha^\top M \alpha.
	\end{multline*}
	Since $\innerprod{\alpha}{z}+\tfrac12\| z\|^2 = \tfrac12\|z+\alpha\|^2-\tfrac12\|\alpha\|^2$, this is \eqref{eq:ident1}.
	
	For $y=x$ we have $g(x)=\varphi(x)$, $z=0$, $\Lambda(y,\alpha)=\alpha$, so \eqref{eq:ident1} gives $g(x) = \Phi(x;w(\alpha))-\tfrac12\alpha^\top M\alpha$.
	Subtracting this from \eqref{eq:ident1} yields \eqref{eq:ident2}.
\end{proof}

\subsection*{Proof of \cref{thm:kleinmichel_compact}}
\begin{proof}
	We proceed by induction.
	When $k=1$ the right-hand side of \eqref{eq_KM_C} is 
	\[
	\overline{\gamma}_1 H_0 + Q_1 M_1^{-1}Q_1^\top
	=
	\gamma_0 H_0 + \frac{(y^0 - \gamma_0 H_0 d^0) (y^0 - \gamma_0 H_0 d^0)^\top}{\innerprod{y^0- \gamma_0 H_0 d^0}{d^0}}
	=
	H_1.
	\]
	Let us now assume that \eqref{eq_KM_C} holds for some $k$. 
	Moreover, define
	\[
	\delta_k \coloneqq \innerprod{q^k}{d^k} - \gamma_k (Q_k^\top d^k)^\top M_k^{-1} Q_k^\top d^k.
	\]
	By definition of $Q_k$, $\overline{\gamma}_k$ and the formula for $H_k$, it holds that
	\begin{multline*}
		\delta_k
		=
		\innerprod{q^k}{d^k} - \gamma_k (Q_k^\top d^k)^\top M_k^{-1} Q_k^\top d^k
		=
		\innerprod{y^k}{d^k} - \overline{\gamma}_{k+1} \innerprod{d^k}{H_0 d^k} - \gamma_k (Q_k^\top d^k)^\top M_k^{-1} Q_k^\top d^k
		\\
		=
		\innerprod{y^k}{d^k} - \gamma_k \innerprod{d^k}{\left( \overline{\gamma}_k H_0 + Q_k M_k^{-1} Q_k^\top \right) d^k}
		=
		\innerprod{y^k - \gamma_k H_k d^k}{d^k}
		\neq 0
	\end{multline*}
	by assumption.
	Applying the Kleinmichel update to $H_k$ we have
	\begin{align*}
		H_{k+1}
		={}&
		\gamma_k \left( \overline{\gamma}_k H_0 + Q_k M_k^{-1} Q_k^\top \right) + \frac{(q^k - \gamma_k Q_k M_k^{-1} Q_k^\top d^k)(q^k - \gamma_k Q_k M_k^{-1} Q_k^\top d^k)^\top}{\innerprod{q^k}{d^k} - \gamma_k \innerprod{d^k}{Q_k M_k^{-1}Q_k^\top d^k}}
		\\
		={}&
		\overline{\gamma}_{k+1} H_0 + \frac{1}{\delta_k} \left[ \delta_k \gamma_k Q_k M_k^{-1} Q_k^\top + q^k (q^k)^\top - \gamma_k q^k (d^k)^\top Q_k M_k^{-1} Q_k^\top
		\right.\\ &\left.\qquad
		- \gamma_k Q_k M_k^{-1} Q_k^\top d^k (q^k)^\top + \gamma_k^2 (Q_k M_k^{-1}Q_k^\top d^k)(Q_k M_k^{-1}Q_k^\top d^k)^\top \right]
		,
	\end{align*}
	which can be expressed as
	\begin{multline}
		\label{eq_km_compact_1}
		H_{k+1}
		= \overline{\gamma}_{k+1} H_0
		\\ 
		+ \frac{1}{\delta_k} \begin{bmatrix} Q_k & q^k \end{bmatrix} \begin{bmatrix}
			\gamma_k M_k^{-1} (\delta_k I + \gamma_k Q_k^\top d^k (d^k)^\top Q_k M_k^{-1}) & - \gamma_k M_k^{-1} Q_k^\top d^k \\
			- \gamma_k (d^k)^\top Q_k M_k^{-1} & I
		\end{bmatrix}
		\begin{bmatrix}
			Q_k^\top \\
			(q^k)^\top
		\end{bmatrix} .
	\end{multline}
	By direct multiplication we obtain
	\begin{equation}
		\label{eq_km_compact_2}
		\begin{bmatrix}
			\gamma_k^{-1} M_k & Q_k^\top d^k \\
			(d^k)^\top Q_k & \innerprod{q^k}{d^k}
		\end{bmatrix} 
		\begin{bmatrix}
			\gamma_k M_k^{-1} (\delta_k I + \gamma_k Q_k^\top d^k (d^k)^\top Q_k M_k^{-1}) & - \gamma_k M_k^{-1} Q_k^\top d^k \\
			- \gamma_k (d^k)^\top Q_k M_k^{-1} & I
		\end{bmatrix} \frac{1}{\delta_k} = I,
	\end{equation}
	since
	\begin{multline*}
		\gamma_k (d^k)^\top Q_k M_k^{-1} \big(\delta_k I + \gamma_k Q_k^\top d^k (d^k)^\top Q_k M_k^{-1} \big) - \gamma_k (q^k)^\top d^k (d^k)^\top Q_k M_k^{-1}
		\\
		=
		\gamma_k \big(\delta_k + \gamma_k (d^k)^\top Q_k M_k^{-1} Q_k^\top d^k - \innerprod{q^k}{d^k} \big) (d^k)^\top Q_k M_k^{-1}
		= 0
	\end{multline*}
	by definition of $\delta_k$.
	Therefore, $M_{k+1}$ is invertible with $M_{k+1}^{-1}$ given by the second matrix in \eqref{eq_km_compact_2}.
	However, this is also the matrix appearing in \eqref{eq_km_compact_1} and hence we see that \eqref{eq_km_compact_1} is exactly \eqref{eq_KM_C} with $k+1$ instead of $k$, which establishes the result.
\end{proof}

\section*{Declarations}

\paragraph{Data availability}
The test problems in \cref{Sec:Numerics} are based on randomly generated data.
The code used for data generation and numerical experiments is archived on Zenodo at \textsc{doi}:~\href{https://doi.org/10.5281/zenodo.20025657}{10.5281/zenodo.20025657}.

\paragraph{Conflict of interest}
The authors declare that they have no conflict of interest to this work.

\phantomsection
\addcontentsline{toc}{section}{References}%
{\small%
\bibliographystyle{habbrv}%
\bibliography{literature}%
}

\end{document}